\documentclass{svjour3}
\journalname{}
\usepackage{amsmath,amsfonts,amssymb}
\usepackage{graphicx}
\usepackage{tikz-cd}
\usepackage{multirow}

\begin{document}
\title{Implicit-explicit all-speed schemes for compressible Cahn-Hilliard-Navier-Stokes
  equations}
\author{Andreu Martorell, Pep Mulet, Dionisio F. Yáñez}
\institute{
  Department of Mathematics,  Universitat de
  València    (Spain); email: mulet@uv.es.
}
\maketitle

\begin{abstract}
    We propose a second-order implicit-explicit (IMEX) time-stepping scheme for the isentropic, compressible Cahn-Hilliard-Navier-Stokes equations in the low Mach number regime.
    The method is based on finite differences on staggered grids and is specifically designed to handle the challenges posed by the low Mach number limit, where the system approaches to an incompressible behavior. 
    In this regime, standard explicit schemes suffer from severe time-step restrictions due to fourth-order diffusion terms and the stiffness induced by fast acoustic waves.
    To overcome this, we employ an IMEX strategy which splits the governing equations into stiff and non-stiff components.
    The stiff terms, arising from pressure, viscous forces and fourth-order Cahn-Hilliard contributions, are treated implicitly, while the remaining are dealt explicitly.
    \keywords{Asymptotic preserving, low Mach number, implicit-explicit schemes, incompressible limit, isentropic compressible Cahn-Hilliard-Navier-Stokes}
\end{abstract}

\section{Introduction}

In fluid dynamics, the Cahn-Hilliard (CH) equation describes the phenomenon of phase separation in two-phase systems \cite{CahnHilliard59}. 
It captures the temporal evolution of a mixture of immiscible fluids through the formation of diffuse interfaces between the phases. 
This diffusive-interface approach is well-suited for explaining topological changes, such as layered structures in sedimenting colloidal suspensions \cite{Siano79}.

However, the classical CH model does not take into account the fluid dynamics such as motion of the medium, viscosities, or external forces like gravity.
To incorporate these effects, the CH system is combined with the Navier-Stokes equations, which describe the conservation of mass and momentum balance of fluids. 
The resulting system of partial differential equations, known as the Cahn-Hilliard-Navier-Stokes (CHNS), provides a thermodynamically consistent framework for modeling multiphase flows \cite{AbelsFeireisl08,Boyer99,LT98}. 
The CHNS model has been applied to a wide variety of problems: separation of immiscible fluids, bubble dynamics or the evolution of interfaces in multiphase systems \cite{Kynch52,Siano79}.

In many practical applications, the flow occurs in a low Mach number regime.
The Mach number, defined as the ratio between the characteristic fluid velocity and the speed of sound, becomes small in these situations.
Under these conditions, the flow behaves as a nearly incompressible \cite{abels_liu_necasova_24,Alazard06,LM98}, but compressible effects are still present.

Designing efficient numerical schemes for the compressible CHNS equations for all Mach number regimes presents several challenges. 

First, as the squared low Mach number $\delta$ tends to zero, the system becomes increasingly stiff due to the presence of fast acoustic waves with characteristic speeds, namely, $\mathbf{v} \pm \frac{1}{\delta}c,$ where $\mathbf{v}$ is the velocity filed and $c$ represents the speed of sound.
This stiffness is further incremented by the presence of up to fourth-order spatial derivatives in the CHNS equations leading to discrete operators with large eigenvalues. 
As a consequence, explicit solvers are severely constrained by stability conditions,
$$\Delta t \approx \mathcal{O}\left(\Delta x^4\right) + \mathcal{O}\left(\frac{\delta}{\max|\delta\mathbf{v}\pm c|}\right),$$ 
where $\Delta t$ and $\Delta x$ are the time and spatial step size, respectively.
In addition, when solving ODE systems of the form $z^\prime=f(z)$, the appearance of stiff, negative-definite Jacobians allows implicit methods to take larger time steps, whereas explicit methods are limited to $\Delta t\approx\mathcal{O}(|\lambda|^{-1})$,
with $|\lambda|$ the largest absolute magnitude of the Jacobian.

Another challenge arises from the intrinsic structure of the CH equation, whose gradient-flow nature is associated with a non-convex energy functional \cite{LT98}. 
To ensure stability, a common strategy is to decompose the energy into the difference of two differentiable convex functions. 
By treating the contractive part implicitly and the expansive part explicitly, one can construct unconditionally stable IMEX scheme, as shown in \cite{Eyre98,Vollmayr-Lee-Rutenberg2003}. 

One way of dealing with these difficulties is to split the pressure into stiff and non-stiff components and treating the stiff part implicitly (for e.g. \cite{cordier_degond_kumbaro_12,degond_tang_11}).
Therefore, the time step is no longer constrained by the Mach number. 
Similarly, the stiff fourth-order CH terms are treated implicitly, while the remaining are handled explicitly \cite{MMY25,mulet_24}.
To maintain stability and preserve symmetry, one can discretize the variables on staggered grids \cite{harlow_welch_65,Patankar18}. 

As shown in \cite{abels_liu_necasova_24}, when the Mach number tends to zero, the solutions of the compressible CHNS equations converge to those of its incompressible counterpart. 
Building on this, numerical methods for the compressible CHNS equations are proposed in \cite{HeShi20,MMY25,mulet_24}, while for the quasi-incompressible and incompressible CHNS models are presented in \cite{ChenZhao20,HanWang15,Jia20,LiXu21}.
Numerical schemes for the low Mach number of the compressible Euler equations are discussed in \cite{cordier_degond_kumbaro_12,degond_tang_11,Haack-Jin-Liu2012,NBALMM14}.

The goal of this work is to develop a second-order IMEX Runge-Kutta scheme for the isentropic, compressible CHNS equations {in low Mach}{for all Mach} number regimes.
The method, based on finite differences on staggered grids, is designed to be an Asymptotic Preserving (AP) scheme \cite{degond_tang_11}, meaning that its stability and accuracy are independent of the Mach number and that it correctly captures the incompressible limit as $\delta\to0$.

The outline of the current work is organized as follows. 
In Section \ref{section_CHNS} we introduce the isentropic, compressible CHNS equations in a low Mach number regime.
Section \ref{section_numerical_schemes} presents a partitioned IMEX Runge-Kutta scheme for the two-dimensional case, and in Section \ref{section_AP} its AP property is proven.
In Section \ref{section_numerical_experiments}, numerical experiments are performed in order to verify the stability, accuracy, and efficiency of our proposed scheme.
Finally, Section \ref{section_conlusion} summarizes the main conclusions and discusses the future directions of our research.

\section{Cahn-Hilliard-Navier-Stokes Equations}\label{section_CHNS}
\subsection{Model Description}
We follow the model from \cite{AbelsFeireisl08,abels_liu_necasova_24} describing the dynamics of two immiscible, compressible, viscous Newtonian fluids in a bounded, open domain $\Omega \subset \mathbb{R}^3$. 
Let $\rho_i$, $c_i$, $\mathbf{v}_i$ denote, respectively, the density, mass concentration, velocity $\mathbf{v}_i$ of the fluid $i=1,2$.
The mixture density is $\rho$ and the barycentric velocity, $\mathbf{v}$, is $\rho\mathbf{v} = \rho_1\mathbf{v}_1 + \rho_2\mathbf{v}_2$.
The concentration difference $c = c_1 - c_2$, taking values in $[-1,1]$, serves as an order parameter distinguishing the two fluid components.

The total Helmholtz free energy of the system is defined as
$$\mathcal{E}(\rho, c)=\int_\Omega\!\left(\rho f(\rho, c) + \frac{\varepsilon}{2}|\nabla c|^2\right)\ d\mathbf{x},$$
where the positive parameter $\varepsilon$ is related to the diffuse interface thickness, which controls the width of the transition region between the two phases.
The term $\frac{\varepsilon}{2}|\nabla c|^2$ represents the classical Cahn-Hilliard regularizing term \cite{LT98}. 
The specific Helmholtz free energy is assumed to have the form
$$f(\rho, c) = f_e(\rho) + \psi(c),$$
where $f_e$ is the potential energy and 
\begin{equation}\label{eq_poly_double_well}
    \psi(c) = \frac{1}{4}\left(c^2-1\right)^2,
\end{equation} 
is a double-well potential.
The thermodynamic pressure is related to the potential energy through the relation $p(\rho)=\rho^2\frac{\partial f(\rho, c)}{\partial\rho}$.
For an isentropic process, we adopt 
$$f_e(\rho)=C_p\frac{\rho^{\gamma-1}}{\gamma-1},\quad\text{so that}\quad p(\rho) = C_p\rho^\gamma,$$ 
for a positive constant $C_p$ and $\gamma>1$ is the adiabatic constant.

The evolution of the mixture is governed by the isentropic, compressible Cahn-Hilliard-Navier-Stokes with gravitational force,
\begin{equation}\label{compressible_CHNS}
  \left\{
  \begin{split}
    &\rho_t + \operatorname{div}(\rho \mathbf{v}) = 0,\\[2pt]
    &(\rho\mathbf{v})_t + \operatorname{div}(\rho\mathbf{v}\otimes\mathbf{v}) + \nabla p = \rho\mathbf{g} + \operatorname{div}\mathbb{T},\\[2pt]
    &(\rho c)_t + \operatorname{div}(\rho c\mathbf{v}) = \Delta\mu,
  \end{split}
  \right.
\end{equation}
where the operators $\operatorname{div}$ and $\Delta$ represents the divergence and laplacian operators, respectively.
The first equation expresses the conservation of total mass of the mixture, the second equation represents the balance of momentum taking into account the gravitation acceleration $\mathbf{g}$, and the third is a Cahn–Hilliard-type equation modeling the evolution of the concentration difference $c$.
The viscous and capillary effects \cite{Ericksen91} are incorporated into the model through the stress tensor $\mathbb{T} = \mathbb{T}_1 + \mathbb{T}_2$ with
$$
    \mathbb{T}_1(\mathbf{v}) = \nu(\nabla\mathbf{v} + \nabla\mathbf{v}^T) + \lambda\operatorname{div}\mathbf{v}\mathbb{I},
    \quad
    \mathbb{T}_2(c) = \frac{\varepsilon}{2}|\nabla c|^2\mathbb{I} - \varepsilon(\nabla c\otimes\nabla c),
$$
where $\lambda$ and $\nu$ are the viscosity coefficients, which are assumed to be positive. 
The chemical potential is defined as the variational derivative of the total free energy with respect to the order parameter $c$. In this setting, it is given by
$$\rho\mu = \frac{\delta\mathcal{E}}{\delta c} = \rho\frac{\partial f}{\partial c} - \varepsilon\Delta c.$$

The system \eqref{compressible_CHNS} is supplemented with initial conditions for the density, velocity, and concentration, given by $\rho(0,x) = \rho_0(x)$, $\mathbf{v}(0,x) = \mathbf{v}_0(x)$, and $c(0,x) = c_0(x)$.
In addition, we impose the boundary conditions
\begin{equation}\label{bdry_cond}
    \mathbf{v}\left.\right|_{\partial\Omega} = \nabla c\cdot\mathbf{n}\left.\right|_{\partial\Omega} = \nabla \mu\cdot\mathbf{n}\left.\right|_{\partial\Omega} = 0,
\end{equation}
where $\mathbf{n}$ denotes the outward unit normal vector to the boundary $\partial\Omega$.

In \cite[Theorem 1.2]{AbelsFeireisl08}, it was proven that, for $\gamma > \frac{3}{2}$ and suitable initial data $(\rho_0, \mathbf{v}_0, c_0)$, system \eqref{compressible_CHNS} with boundary conditions \eqref{bdry_cond} admits global-in-time weak solutions in the sense of Di Perna and Lions.

\subsection{Isentropic compressible Cahn-Hilliard-Navier-Stokes in a low Mach number regime}
In the present work, we focus on the low Mach number regime, which corresponds to taking $C_p \gg 0$. 
In such regimes, the pressure becomes extremely large, and in many numerical methods (e.g., \cite{MMY25,mulet_24}), $C_p$ appears explicitly in the time step stability restriction, 
$$\Delta t \approx \mathcal{O}\left(\sqrt{C_{p}^{-1}}\Delta x\right),$$
where $\Delta t$ and $\Delta x$ are the time and spacial step sizes, respectively, leading to severe restrictions on the time step. 

To this end, we split the pressure into a stiff and non-stiff component, treating the stiff part implicitly \cite{cordier_degond_kumbaro_12,degond_tang_11}. 
Specifically, we write
\begin{equation}\label{pressure_splitting}
  p(\rho) = C_{p,1} \rho^\gamma + C_{p,2} \rho^\gamma, \quad \text{with} \quad C_{p,1} + C_{p,2} = C_{p}, \quad C_{p,2} \gg C_{p,1}.
\end{equation}
We denote $p_1(\rho) = C_{p,1} \rho^\gamma$ and $p_2(\rho) = C_{p,2}\rho^\gamma$ with the squared Mach number defined as $\delta = C_{p}^{-1}$.

In general, the choice of the pressure splitting depends on the characteristic fluid speed. 
For working in a hyperbolic framework and suppressing nonphysical oscillations, we consider $C_{p,1}>0$. 
However, for flows with stronger shocks, a larger value of $C_{p,1}$ is necessary, corresponding to an almost fully explicit treatment of the pressure \cite{degond_tang_11}. 
In the current work, the effects of this choice are investigated numerically in the section of numerical experiments.

We consider both the one- and two-dimensional setting. 
For the latter case, we denote the velocity field by $\mathbf{v} = (v_1, v_2)$ and the gravity $\mathbf{g}=(0,g)$ acting only on the vertical axes.
Under this assumption, the governing equations reduce to the following two-dimensional form:
\begin{equation}\label{eq_compressible_chns_2D}
	\begin{split}
		 {\rho_t} + 
		 {(\rho v_1)_x + (\rho v_2)_y} =&\enspace 0, \\[8pt]
		 {(\rho v_1)_t} + 
		 {(\rho v_1^2 + p_1(\rho))_x + (\rho v_1 v_2)_y}
		=&\enspace  - (p_2(\rho))_x + {\frac{\varepsilon}{2}(c_y^2 - c_x^2)_x - \varepsilon (c_x c_y)_y} +\\[5pt] 
		&\quad  {\nu \Delta v_1 + (\nu + \lambda)((v_1)_{xx} + (v_2)_{xy})}, \\[8pt]
		 {(\rho v_2)_t} + 
		 {(\rho v_2^2 + p_1(\rho))_y + (\rho v_1 v_2)_x}
		=&\enspace  {\rho g} - (p_2(\rho))_y +  {\frac{\varepsilon}{2}(c_x^2 - c_y^2)_y - \varepsilon (c_x c_y)_x} +\\[5pt] 
		&\quad {\nu \Delta v_2 + (\nu + \lambda)((v_1)_{xy} + (v_2)_{yy})}, \\[8pt]
		 {(\rho c)_t} + 
		 {(\rho c v_1)_x + (\rho c v_2)_y} =&\enspace  {\Delta \left( \psi'(c) - \frac{\varepsilon}{\rho} \Delta c \right)}.
	\end{split}
\end{equation}
For the one-dimensional case the system reads as follows:
\begin{equation}\label{eq_compressible_chns_1D}
	\begin{split}
		 {\rho_t} + 
		 (\rho v)_x =&\enspace 0, \\[8pt]
		 {(\rho v)_t} + {(\rho v^2 + p_1(\rho))_x}
		=&\enspace  - (p_2(\rho))_x + \rho g + \left(\left(2\nu+\lambda\right)v_{x} - \frac{\varepsilon}{2}c^2_x\right)_x\\[8pt] 
		 {(\rho c)_t} + 
		 (\rho c v)_x =&\enspace  \left( \psi'(c) - \frac{\varepsilon}{\rho} \Delta c \right)_{xx}.
	\end{split}
\end{equation}

\section{Numerical Schemes}\label{section_numerical_schemes}

\subsection{Spatial Semidiscretization}
We consider the compressible, isentropic CHNS equations in two-spatial dimensions \eqref{eq_compressible_chns_2D} on the square domain $\Omega=[0,1]^2$.
The computational grid is based on a MAC approach \cite{harlow_welch_65}.
Let $\mathbf{x}=(x,y)$ denote the spatial variable. 
The cell-centered grid consists of $M^2$ nodes
$$
    \mathbf{x}_{i,j} = \left(x_i, y_j\right)
    \quad\text{where}\quad 
    x_{i}=\left(i-\tfrac{1}{2}\right)h,\quad y_{j} = \left(j-\tfrac{1}{2}\right)h,
$$
for $i,j=1,\cdots,M$, and with uniform mesh size $h = \frac{1}{M}$. 
The staggered (dual) grid consists in $2M(M-1)$  nodes:
\begin{equation*}
    \begin{split}
        &\mathbf{x}_{i+\frac{1}{2},j},\quad\text{for}\quad j = 1,\cdots, M,\enspace i=1,\cdots,M-1,\\
        &\mathbf{x}_{i,j+\frac{1}{2}},\quad\text{for}\quad i = 1,\cdots, M,\enspace j=1,\cdots,M-1,
    \end{split}
\end{equation*}
where $x_{i+\frac{1}{2}} = ih$ and $y_{j+\frac{1}{2}} = jh$.
The continuity and the Cahn-Hilliard type equations are treated at cell centers, while the momentum equations are treated at the dual points: the horizontal component on vertical cell faces and the vertical component on horizontal cell faces.
To compute momentum at staggered points, we used local averages of the density. 
For instance, for the horizontal momentum we define $(\rho_{*,x})_{i,j}=\rho_{*,i+\frac{1}{2},j} = \frac{1}{2}\left(\rho_{i+i,j} + \rho_{i,j}\right)$ so that 
$$(\rho v_1)(\mathbf{x}_{i+\frac{1}{2},j}) = \rho_{*,i+\frac{1}{2},j}v_{1, i+\frac{1}{2},j},$$
for $i=1,\cdots,M-1$, $j=1,\cdots,M$. 
Notice that $v_{1, i+\frac{1}{2},j} = (\rho v_1)_{i+\frac{1}{2},j}/\rho_{*,i+\frac{1}{2},j}$.
We denote by $\rho_*\mathbf{v} = (\rho_{*,x}v_1,\rho_{*,y}v_2)$. 
For the velocity components, we assume no-slip boundary conditions, i.e.
$$v_{1,\frac{1}{2},j} = v_{1,M+\frac{1}{2},j} = 0,\text{ and } v_{1,i,\frac{1}{2}} = v_{1,i,M+\frac{1}{2}} = 0.$$
For the points outside the wall, we assume symmetric reflection for the density and odd reflection for the velocity components. 
For instance,
\begin{equation*}
    \begin{aligned}
        \rho_{i,j} &= \rho_{|i|+1,j},        &\qquad 
        v_{1,i-1/2,j} &= -\,v_{1,|i|+1+1/2,j}, 
        & i &\in \{0,-1,-2,\ldots\},\\[4pt]
        \rho_{i,j} &= \rho_{i-1,j},          &\qquad 
        v_{1,i+1/2,j} &= -\,v_{1,i-1/2,j},    
        & i &\in \{M+1,M+2,M+3,\ldots\},
    \end{aligned}
\end{equation*}
for $j=1,\ldots,M$. 
Similarly, it is done in the other direction.

After applying this spatial semi-discretization to \eqref{eq_compressible_chns_2D}, one needs to solve a system of $N=2M^2+2M(M-1)$ ordinary differential equations given by
\begin{equation}\label{semi_discrete_system}
    \begin{split}
        U^\prime(t) &= \mathcal{L}(U(t))\\
        U(0) &= U_0,
    \end{split}
\end{equation}
where $U_0$ is the vector of initial conditions and the unknown variables $U = (u_k)_{k=1}^4$ are
$$\rho_{i,j} = u_{1,i,j},\quad (\rho v_1)_{i+\frac{1}{2},j} = u_{2,i+\frac{1}{2},j} \quad (\rho v_2)_{i,j+\frac{1}{2}} = u_{3,i,j+\frac{1}{2}},\quad (\rho c)_{i,j} = u_{4,i,j},$$
for $i,j$ running over their respective grid indices.
Dropping the time dependence of $U$, $\mathcal{L}(U)$ is the nonlinear operator storing the spatially discretized differential operators,
$$\mathcal{L}(U)= {\mathcal{C}_1}(U) + {\mathcal{C}_2}(U)+{\mathcal{L}}_1(U)+{\mathcal{L}}_2(U)+{\mathcal{L}}_3(U)+{\mathcal{L}}_4(U),$$ 
where nonzero terms of operators above are defined as follows:
\begin{equation*}
    \begin{split}
        {\mathcal{C}_1}(U)_{1,i,j}&\approx -(      (\rho v_1)_x+(\rho v_2)_{y})(\mathbf{x}_{i,j}, t),\\
        {\mathcal{C}_2}(U)_{2,i+\frac{1}{2},j}&\approx -((\rho v_1^2 + p_1(\rho))_{x}+(\rho v_1 v_2)_{y})(\mathbf{x}_{i+\frac{1}{2},j}, t),\\              
        {\mathcal{C}_2}(U)_{3,i,j+\frac{1}{2}}&\approx -((\rho v_1v_2)_{x}+(\rho v_2^2+p_1(\rho))_{y})(\mathbf{x}_{i,j+\frac{1}{2}}, t),\\  
        {\mathcal{C}_2}(U)_{4,i,j}&\approx -(       (\rho c v_1)_{x}+(\rho c v_2)_{y})(\mathbf{x}_{i,j}, t),\\
        {\mathcal{L}}_1(U)_{2,i+\frac{1}{2},j}
            &\approx\left(- \left(p_2(\rho)\right)_x\right)((\mathbf{x}_{i,j+\frac{1}{2}},t)),\\
        {\mathcal{L}}_1(U)_{3,i,j+\frac{1}{2}}
            &\approx\left(\rho g - \left(p_2(\rho)\right)_y\right)((\mathbf{x}_{i,j+\frac{1}{2}},t)),\\
        {\mathcal{L}}_2(U)_{2,i+\frac{1}{2},j}&\approx 
            \varepsilon(
            \frac{1}{2}(c_{y}^2)_{x}-\frac{1}{2} (c_{x}^2)_x- (c_{x}c_y)_{y})(\mathbf{x}_{i+\frac{1}{2},j}, t),\\
        {\mathcal{L}}_2(U)_{3,i,j+\frac{1}{2}}&\approx 
            \varepsilon(      \frac{1}{2}(c_{x}^2)_{y}-\frac{1}{2} (c_{y}^2)_y    -(c_xc_y)_x)(\mathbf{x}_{i,j+\frac{1}{2}}, t),\\
        {\mathcal{L}}_{3}(U)_{4,i,j}
            &\approx\Delta(\psi'(c) - \frac{\varepsilon}{\rho}  \Delta c)  (\mathbf{x}_{i,j}, t),\\
        {\mathcal{L}}_4(U)_{2,i+\frac{1}{2},j}&\approx
            (\nu(  (v_1)_{xx}+(v_1)_{yy})+(\nu+\lambda)(  (v_1)_{xx}+(v_2)_{xy}))(\mathbf{x}_{i+\frac{1}{2},j}, t),\\
        {\mathcal{L}}_4(U)_{3,i,j+\frac{1}{2}}&\approx
            (\nu((v_2)_{xx}+(v_2)_{yy})  +(\nu+\lambda)(  (v_1)_{xy}+(v_2)_{yy}))(\mathbf{x}_{i,j+\frac{1}{2}}, t).
    \end{split}
\end{equation*}

\subsubsection{Basic Finite Difference Operators}
In this section, we introduce the finite difference operators used to approximate the spatial derivatives of the system \eqref{eq_compressible_chns_2D} on MAC grids.

To approximate the first derivatives at the grid points $\mathbf{x}_{i,j}$, we employ central differences which are second-order accurate at interior points ($1 < i,j < M$), and first-order otherwise satisfying the boundary conditions \eqref{bdry_cond}. 
The resulting discrete derivative operator in one spatial direction can be written in matrix form as
\begin{equation*}
D_M^c = \frac{1}{2h}
\begin{bmatrix}
-1 & 1 & 0 & \dots & 0\\
-1 & 0 & 1 & \dots & 0\\
\vdots & \ddots & \ddots & \ddots & \vdots\\
0 & \dots & -1 & 0 & 1\\
0 & \dots & 0 & -1 & 1
\end{bmatrix} \in \mathbb{R}^{M \times M}.
\end{equation*}

For the dual grid, we define two finite difference matrices $D_M$ and $D_M^\ast$ of size $M \times (M-1)$ that approximate first derivatives at the cell interfaces 
$\mathbf{x}_{i+\frac{1}{2},j}$ or $\mathbf{x}_{i,j+\frac{1}{2}}$ which are second-order accurate at interior points and first-order otherwise. Both matrices incorporate the appropriate boundary conditions \eqref{bdry_cond}:
\begin{equation*}
D_M = \frac{1}{h}
\begin{bmatrix}
1 & 0 & 0 & \dots & 0\\
-1 & 1 & 0 & \dots & 0\\
\cdots & \cdots & \cdots & \cdots & \cdots\\
0 & \dots & 0 & -1 & 1\\
0 & \dots & 0 & 0 & -1
\end{bmatrix}, 
\quad
D_M^{\ast} = \frac{1}{h}
\begin{bmatrix}
2 & 0 & 0 & \dots & 0\\
-1 & 1 & 0 & \dots & 0\\
\cdots & \cdots & \cdots & \cdots & \cdots\\
0 & \dots & 0 & -1 & 1\\
0 & \dots & 0 & 0 & -2
\end{bmatrix}.
\end{equation*}

We also define the averaging matrix $A_M \in \mathbb{R}^{(M-1) \times M}$, which is used to interpolate quantities between the cell centers and the staggered grid:
\begin{equation*}
A_M = \frac{1}{2}
\begin{bmatrix}
1 & 1 & 0 & \dots & 0\\
0 & 1 & 1 & \dots & 0\\
\vdots & \ddots & \ddots & \ddots & \vdots\\
0 & \dots & 0 & 1 & 1
\end{bmatrix}.
\end{equation*}

We introduce $f\ast g = (f_{i,j}g_{i,j})_{i,j}$ for matrices $f$, $g$ in $\mathbb{R}^{n\times m}$.

\subsubsection{The Operators $\mathcal{C}_1$ and $\mathcal{C}_2$}
The convective part of the system is decomposed into two operators,
$$\mathcal{C}(U) = \mathcal{C}_1(U) + \mathcal{C}_2(U),$$
where $\mathcal{C}_1$ acts only on the continuity equation and $\mathcal{C}_2$ on the momentum and Cahn-Hilliard type equations.\\

We consider the fluxes in the $x$- and $y$-directions:
\begin{small}
\begin{equation*}
F(U) =
\begin{bmatrix}
F^{\rho} \\[3pt]
F^{\rho v_1} \\[3pt]
F^{\rho v_2} \\[3pt]
F^{c}
\end{bmatrix}
=
\begin{bmatrix}
\rho v_1 \\[3pt]
\rho v_1^2 + p_1(\rho) \\[3pt]
\rho v_1 v_2 \\[3pt]
\rho v_1 c
\end{bmatrix},
\qquad
G(U) =
\begin{bmatrix}
G^{\rho} \\[3pt]
G^{\rho v_1} \\[3pt]
G^{\rho v_2} \\[3pt]
G^{c}
\end{bmatrix}
=
\begin{bmatrix}
\rho v_2 \\[3pt]
\rho v_1 v_2 \\[3pt]
\rho v_2^2 + p_1(\rho) \\[3pt]
\rho v_2 c
\end{bmatrix}.
\end{equation*}
\end{small}
Let $\hat{F}^{\ast}$ and $\hat{G}^{\ast}$ denote the numerical fluxes associated to $F^{\ast}$ and $G^{\ast}$, respectively.
The convective operators are approximated using numerical flux differences at cell centers and cell interfaces.
Specifically, 
\begin{equation*}
    \begin{split}
        \mathcal{C}_1(U)_{1,i,j} &\approx -\frac{\hat{F}^{\rho}_{i+\frac{1}{2},j} - \hat{F}^{\rho}_{i-\frac{1}{2},j}}{h} -\frac{\hat{G}^{\rho}_{i,j+\frac{1}{2}} - \hat{G}^{\rho}_{i,j-\frac{1}{2}}}{h},\\
        \mathcal{C}_2(U)_{2,i+\frac{1}{2},j} &\approx -\frac{\hat{F}^{\rho v_1}_{i+1,j} - \hat{F}^{\rho v_1}_{i,j}}{h} -\frac{\hat{G}^{\rho v_1}_{i+\frac{1}{2},j+\frac{1}{2}} - \hat{G}^{\rho v_1}_{i+\frac{1}{2},j-\frac{1}{2}}}{h},\\
        \mathcal{C}_2(U)_{3,i,j+\frac{1}{2}} &\approx -\frac{\hat{F}^{\rho v_2}_{i+\frac{1}{2},j+\frac{1}{2}} - \hat{F}^{\rho v_2}_{i-\frac{1}{2},j+\frac{1}{2}}}{h} -\frac{\hat{G}^{\rho v_2}_{i,j+1} - \hat{G}^{\rho v_2}_{i,j}}{h},\\
        \mathcal{C}_2(U)_{4,i,j} &\approx -\frac{\hat{F}^{c}_{i+\frac{1}{2},j} - \hat{F}^{c}_{i-\frac{1}{2},j}}{h} -\frac{\hat{G}^{c}_{i,j+\frac{1}{2}} - \hat{G}^{c}_{i,j-\frac{1}{2}}}{h}.
    \end{split}
\end{equation*}
The numerical fluxes are computed using the Rusanov flux.
For the explicit terms of the operator $\mathcal{C}$ (see Section \ref{section_imex_schemes}), we use WENO5 reconstructions, which are fifth-order accurate for finite difference schemes \cite{BBMZ19a,BBMZ19b,PMP19}.
Let us describe it for the $x$-direction case.
Denote by $\mathcal{W}^{x}:\mathbb{R}^5\longrightarrow\mathbb{R}$ the WENO5 reconstruction operator and a function $f$ such that $f_{i,j} = f(\mathbf{x}_{i,j})$ for indexes $i$, $j$ running both in the primal and dual grids. 
So the right and left state reconstructions are, respectively, for primal grids,
\begin{equation*}
  f^{+}_{i+1,j} = \mathcal{W}^{x}\left(f_{i+3,j},\ldots,f_{i-1,j}\right),
  \quad
  f^{-}_{i,j} = \mathcal{W}^{x}\left(f_{i-2,j},\ldots,f_{i+2,j}\right),
\end{equation*}
and for dual grids,
\begin{align*}
  f^{+}_{i+\frac12,j}
    &= \mathcal{W}^{x}\left(f_{i+\frac52,j},\ldots,f_{i-\frac32,j}\right), &
  f^{-}_{i-\frac12,j}
    &= \mathcal{W}^{x}\left(f_{i-\frac52,j},\ldots,f_{i+\frac32,j}\right), \\[8pt]
  f^{+}_{i+1,j+\frac12}
    &= \mathcal{W}^{x}\left(f_{i+3,j+\frac12},\ldots,f_{i-1,j+\frac12}\right), &
  f^{-}_{i,j+\frac12}
    &= \mathcal{W}^{x}\left(f_{i-2,j+\frac12},\ldots,f_{i+2,j+\frac12}\right).
\end{align*}
Then, for the terms in $\mathcal{C}_1$,
\begin{align*}
    \hat{F}^{\rho}_{i+\frac{1}{2},j} 
    &=
    \frac{(\rho v_1)_{i+1,j} + (\rho v_1)_{i,j}}{2}
    - \frac{\lambda^{\rho}_{i+\frac{1}{2},j}}{2}\left(\rho^{+}_{i+1,j} - \rho^{-}_{i,j}\right)\\[5pt]
    &= (\rho v_1)_{i+\frac{1}{2},j} 
    - \frac{\lambda^{\rho}_{i+\frac{1}{2},j}}{2}\left(\rho^{+}_{i+1,j} - \rho^{-}_{i,j}\right),
\end{align*}
and for the terms in $\mathcal{C}_2$,
\begin{align*}
    &\hat{F}^{\rho v_1}_{i,j} 
    = \frac{1}{2}\left((\rho v_1^2 + p_1(\rho))^{+}_{i+\frac{1}{2},j} + (\rho v_1^2 + p_1(\rho))^{-}_{i-\frac{1}{2},j}\right)
    - \frac{\lambda^{\rho v_1}_{i,j}}{2}\left((\rho v_1)^{+}_{i+\frac{1}{2},j} - (\rho v_1)^{-}_{i-\frac{1}{2},j}\right),\\[11pt]
    &\hat{F}^{\rho v_2}_{i+\frac{1}{2},j+\frac{1}{2}} 
    = \frac{1}{2}\left((\rho v_1 v_2)^{+}_{i+1,j+\frac{1}{2}} + (\rho v_1 v_2)^{-}_{i,j+\frac{1}{2}}\right) 
    - \frac{\lambda^{\rho v_2}_{i+\frac{1}{2},j+\frac{1}{2}}}{2}\left((\rho v_2)^{+}_{i+1,j+\frac{1}{2}} - (\rho v_2)^{-}_{i,j+\frac{1}{2}}\right),\\[11pt]
    &\hat{F}^{c}_{i+\frac{1}{2},j} 
    = \frac{1}{2}\left((\rho c v_1)^{+}_{i+1,j} + (\rho c v_1)^{-}_{i,j}\right)
    - \frac{\lambda^{c}_{i+\frac{1}{2},j}}{2}\left((\rho c)^{+}_{i+1,j} - (\rho c)^{-}_{i,j}\right).
\end{align*}
The numerical viscosities $\lambda^\ast$ are defined as the maximum of the upper bounds of the local characteristic speeds at the reconstructed states of each $\hat{F}^\ast$, namely,
\begin{equation*}
  \begin{split}
    &\lambda^{\rho}_{i+\frac{1}{2},j} = \lambda^{c}_{i+\frac{1}{2},j} =
    \max\left\{
      \left|v^{+}_{1,i+1,j}\right| + s\left(\rho^{+}_{i+1,j}\right),
      \left|v^{-}_{1,i,j}\right| + s\left(\rho^{-}_{i,j}\right)
    \right\},\\[11pt]
    &\lambda^{\rho v_1}_{i,j} = 
    \max\left\{
      \left|v^{+}_{1,i+\frac{1}{2},j}\right| + s\left(\rho^{+}_{i+\frac{1}{2},j}\right),
      \left|v^{-}_{1,i-\frac{1}{2},j}\right| + s\left(\rho^{-}_{i-\frac{1}{2},j}\right)
    \right\},\\[11pt]
    &\lambda^{\rho v_2}_{i+\frac{1}{2},j+\frac{1}{2}} = 
    \max\left\{
      \left|v^{+}_{1,i+1,j+\frac{1}{2}}\right| + s\left(\rho^{+}_{i+1,j+\frac{1}{2}}\right),
      \left|v^{-}_{1,i,j+\frac{1}{2}}\right| + s\left(\rho^{-}_{i,j+\frac{1}{2}}\right)
    \right\},
  \end{split}
\end{equation*}
where $s(f^{\ast}_{i,j}) = \sqrt{p_1^\prime\left(f^\ast_{i,j}\right)}$ is the speed of sound for indexes $i,j$ in the primal and dual grids.

Notice also that the values of $\rho$ on the dual grid and of $v_1$ on the primal grid are required.
To this end, we employ a sixth-order grid transfer operator defined by the coefficients
$$\mu = \frac{1}{256}\begin{bmatrix}3&-25&150&150&-25&3\end{bmatrix}.$$
Thus,
\begin{equation*}
  \rho_{i+\frac{1}{2},j} = \sum_{k=i-3}^{i+2}\mu_{k+4-i}\rho_{k,j},
  \quad 
  v_{1,i,j} = \sum_{k=i-3}^{i+2}\mu_{k+4-i}v_{1,k+\frac{1}{2},j},
\end{equation*}
and similarly in the $y$-direction.
In particular, to evaluate $\hat{F}^{\rho v_2}$, we must approximate the velocity component $v_1$ at staggered locations where it is not directly defined.
So, we first approximate 
$$v_{1,i+\frac{1}{2},j+\frac{1}{2}} = \frac{v_{1,i+\frac{1}{2},j+1} + v_{1,i+\frac{1}{2},j}}{2},$$
and then apply the previous transfer grid operator in the $x$-direction to obtain $v_{1,i,j+\frac{1}{2}}$.

An analogous procedure is used for the flux $\hat{G}^{\rho v_1}$.

\subsubsection{The Operator $\mathcal{L}_1$}
The nonzero components of the operator $\mathcal{L}_1$ are approximated point-wise and taking central finite differences, specifically,
\begin{equation*}
        \mathcal{L}_1(U)_2 \approx D^T_{M}p_2(\rho),\quad \mathcal{L}_1(U)_3 \approx \rho A_M^T g + p_2(\rho)D_{M}.
\end{equation*}

\subsubsection{The Operator $\mathcal{L}_2$}
The operator $\mathcal{L}_2$ involves the derivatives of the order parameter $c$ in the momentum equation.
For the approximation of $\mathcal{L}_2(U)_2$ we use:
\begin{equation*}
    \begin{split}
        (c^2_x)_x(\mathbf{x}_{i+\frac{1}{2},j}) &\approx - (D^T_M(D^c_Mc*D^c_Mc))_{i,j},\\
        (c^2_y)_x(\mathbf{x}_{i+\frac{1}{2},j}) &\approx - (D^T_M((c(D^c_M)^T)*(c(D^c_M)^T)))_{i,j},\\ 
        (c_xc_y)_y(\mathbf{x}_{i+\frac{1}{2},j}) &\approx- (((D^T_McA_M^T)*(A_McD_M))D^T_{M})_{i,j},
    \end{split}
\end{equation*}
for $i= 1,\cdots,M-1$ and $j=1,\cdots,M$. 

Similarly, for the $\mathcal{L}_2(U)_3$ we use:
\begin{equation*}
    \begin{split}
        (c^2_y)_y(\mathbf{x}_{i,j+\frac{1}{2}}) &\approx - (((c(D^c_M)^T)*(c(D^c_M)^T))D_M)_{i,j},\\
        (c^2_x)_y(\mathbf{x}_{i,j+\frac{1}{2}}) &\approx - ((D^c_Mc*D^c_Mc)D_M)_{i,j},\\ 
        (c_xc_y)_x(\mathbf{x}_{i,j+\frac{1}{2}}) &\approx- (D_M((D^T_McA_M^T)*(A_McD_M)))_{i,j},
    \end{split}
\end{equation*}
for $i= 1,\cdots,M$ and $j=1,\cdots,M-1$. 

These approximations satisfy the boundary conditions \eqref{bdry_cond} and are second-order accurate at interior points and first-order accurate otherwise.

\subsubsection{The Operator $\mathcal{L}_3$}\label{section_eyre_splitting}
The operator $\mathcal{L}_3$, which arises from the Cahn-Hilliard type equation, requires a special treatment. 
This is because, for stability, only negative definite terms should be treated implicitly.
However, the term $\Delta\psi^\prime(c) = \operatorname{div}\left(\psi^{\prime\prime}(c)\nabla c\right)$ changes sign in $(-1,1)$ since the potential $\psi$ is of convex-concave type.
To handle this, in \cite{Eyre98} was shown that if $\psi$ is split into the sum of a convex part $\psi_1$ and a concave part $\psi_2$,
the resulting scheme for the Cahn-Hilliard equation treating $\psi^\prime_1$ implicitly and $\psi^\prime_2$ explicitly is unconditionally stable.  
In particular, we choose
$$\psi^\prime_1(c) = 2c \quad\text{and}\quad \psi^\prime_2(c) = c^3-3c.$$

Let $f\in\mathcal{C}^4$ such that $\nabla f(x,y)\cdot\mathbf{n} = 0$ with $f_{i,j} = f(\mathbf{x}_{i,j})$ for $i,j=1,\cdots, M$.
We use a second-order accurate approximation for $\Delta f_{i,j} \approx \Delta_h f_{i,j} = \Delta_{x,h}f_{i,j} + \Delta_{y,h}f_{i,j}$ where
\begin{equation*}
    \Delta_{x,h}f_{i,j} =
    \begin{cases}
        \frac{f_{i+1,j}-f_{i,j}}{h^2},&\mbox{if }i=1,\\[5pt]
        \frac{f_{i+1,j}-2f_{i,j} + f_{i-1,j}}{h^2},&\mbox{if } 1<i<M,\\[5pt]
        \frac{f_{i-1,j}-f_{i,j}}{h^2},&\mbox{if }i=M,
    \end{cases}
\end{equation*}
and similarly for $\Delta_{y,h}$. 
For $i,j=1,\cdots,M$, yields $$(\Delta\psi^\prime_1(c))(\mathbf{x}_{i,j}) \approx 2\Delta_h c_{i,j},$$
and $(\Delta\psi^\prime_2(c))(\mathbf{x}_{i,j})\approx(\psi^{\prime\prime}_2(c)c_x)_{x}(\mathbf{x}_{i,j}) + (\psi_2^{\prime\prime}(c)c_y)_{y}(\mathbf{x}_{i,j})$ where 
\begin{equation}\label{psi_2_xx}
    (\psi_2^{\prime\prime}(c)c_x)_{x}(\mathbf{x}_{i,j}) \approx
    \begin{cases}
        \frac{(\psi_2^{\prime\prime}(c_{i+1,j})+\psi_2^{\prime\prime}(c_{i,j}))(c_{i+1,j}-c_{i,j})}{2h^2}&i=1,\\[5pt]
        \frac{(\psi_2^{\prime\prime}(c_{i+1,j})+\psi_2^{\prime\prime}(c_{i,j}))(c_{i+1,j}-c_{i,j})-(\psi_2^{\prime\prime}(c_{i,j})+\psi_2^{\prime\prime}(c_{i-1,j}))(c_{i,j}-c_{i-1,j})}{2h^2}&1<i<M,\\[5pt]
        \frac{-(\psi_2^{\prime\prime}(c_{i,j})+\psi_2^{\prime\prime}(c_{i-1,j}))(c_{i,j}-c_{i-1,j})}{2h^2}&i=M,\\
    \end{cases}
\end{equation}
and 
\begin{equation}\label{psi_2_yy}
    (\psi_2^{\prime\prime}(c)c_y)_{y}(\mathbf{x}_{i,j}) \approx
    \begin{cases}
        \frac{(\psi_2^{\prime\prime}(c_{i,j+1})+\psi_2^{\prime\prime}(c_{i,j}))(c_{i,j+1}-c_{i,j})}{2h^2}&j=1,\\[5pt]
        \frac{(\psi_2^{\prime\prime}(c_{i,j+1})+\psi_2^{\prime\prime}(c_{i,j}))(c_{i,j+1}-c_{i,j})-(\psi_2^{\prime\prime}(c_{i,j})+\psi_2^{\prime\prime}(c_{i,j-1}))(c_{i,j}-c_{i,j-1})}{2h^2}&1<j<M,\\[5pt]
        \frac{-(\psi_2^{\prime\prime}(c_{i,j})+\psi_2^{\prime\prime}(c_{i,j-1}))(c_{i,j}-c_{i,j-1})}{2h^2}&j=M,\\    
    \end{cases}
\end{equation}

Now, it only remains to approximate $\Delta\left(\frac{1}{\rho}\Delta c\right)$. 
To this end, we employ the aforementioned second-order accurate approximation for the laplacian, namely,
$$\Delta\left(\frac{1}{\rho}\Delta c\right)(\mathbf{x}_{i,j}) \approx \left(\Delta_hD(\rho)^{-1}\Delta_h c\right)_{i,j},$$
where $D$ is the diagonal operator on $M\times M$ matrices defined as
$$(D(v) w)_{i,j} = v_{i,j} w_{i,j}, \quad i,j = 1,\dots,M, \quad v,w \in \mathbb{R}^{M\times M}.$$

\subsubsection{The Operator $\mathcal{L}_4$}
The operator $\mathcal{L}_4$ stores the derivatives of the velocity field in the balance of momentum. 
For approximating the pure double derivatives, for instance, $(v_1)_{xx}$ and $(v_1)_{yy}$ at $\mathbf{x}_{i+\frac{1}{2}, j}$, we use
\begin{equation*}
    \left(v_1\right)_{xx}\left(\mathbf{x}_{i+\frac{1}{2}, j}\right) =
    \begin{cases}
        \displaystyle\frac{v_{1,i+\frac{3}{2},j} - 2v_{1,i+\frac{1}{2},j}}{h^2},&\mbox{for } i=1,\\[11pt]
        \displaystyle\frac{v_{1,i+\frac{3}{2},j} - 2v_{1,i+\frac{1}{2},j} + v_{1,i-\frac{1}{2},j}}{h^2},&\mbox{for } 1<i<M-1,\\[11pt]
        \displaystyle\frac{v_{1,i-\frac{1}{2},j} - 2v_{1,i+\frac{1}{2},j}}{h^2},&\mbox{for } i=M-1,
    \end{cases}
\end{equation*}
for $j = 1,\cdots, M$, and
\begin{equation*}
    \left(v_1\right)_{yy}\left(\mathbf{x}_{i+\frac{1}{2}, j}\right) =
    \begin{cases}
        \displaystyle\frac{v_{1,i+\frac{1}{2},j+1} - 3v_{1,i+\frac{1}{2},j}}{h^2},&\mbox{for } j=1,\\[11pt]
        \displaystyle\frac{v_{1,i+\frac{1}{2},j+1} - 2v_{1,i+\frac{1}{2},j} + v_{1,i+\frac{1}{2},j-1}}{h^2},&\mbox{for } 1<j<M,\\[11pt]
        \displaystyle\frac{v_{1,i+\frac{1}{2},j-1} - 3v_{1,i+\frac{1}{2},j}}{h^2},&\mbox{for } j=M,
    \end{cases}
\end{equation*}
for $i = 1,\cdots, M-1$.  
The approximation of the cross derivative, e.g., $(v_2)_{xy}$ at $\mathbf{x}_{i+\frac{1}{2},j}$ is given by
\begin{equation*}
    \left(v_2\right)_{xy}\left(\mathbf{x}_{i+\frac{1}{2}, j}\right) =
    \begin{cases}
        \displaystyle\frac{v_{2,i+1, j+\frac{1}{2}} - v_{2,i, j+\frac{1}{2}}}{h^2},&\mbox{for } j=1,\\[11pt]
        \displaystyle\frac{(v_{2,i+1, j+\frac{1}{2}} - v_{2,i, j+\frac{1}{2}}) - (v_{2,i+1, j-\frac{1}{2}} - v_{2,i, j-\frac{1}{2}})}{h^2},&\mbox{for } 1<j<M,\\[11pt]
        \displaystyle\frac{v_{2,i+1, j-\frac{1}{2}} - v_{2,i, j-\frac{1}{2}}}{h^2},&\mbox{for } j=M,
    \end{cases}
\end{equation*}
for $i = 1,\cdots, M-1$. 

The three expressions above verify the boundary conditions \eqref{bdry_cond} and are second-order accurate at its respective interior points and first-order accurate otherwise.

In matrix form, 
\begin{equation}\label{L_4_2}
    \mathcal{L}_4(U)_2\approx -((2\nu + \lambda)D_M^TD_Mv_1 + \nu v_1(D^{\ast}_{M+1})^TD_{M+1} + (\nu + \lambda)D^T_{M}v_2D^T_M).
\end{equation}
Similarly, the other nonzero component of $\mathcal{L}_4$ takes the form
\begin{equation}\label{L_4_3}
    \mathcal{L}_4(U)_3\approx -((2\nu + \lambda)v_2D_M^TD_M + \nu D^T_{M+1}D^{\ast}_{M+1}v_2 + (\nu + \lambda)D_{M}v_1D_M).
\end{equation}

\subsection{Vector Implementation}

In this section, we reformulate system \eqref{semi_discrete_system} in vector form for the two-dimensional case. 
The one-dimensional case follows analogously and is therefore omitted.\\

Let $\operatorname{vec}(A)$ denote the column-wise vectorization of a matrix $A\in\mathbb{R}^{n\times m}$, defined by
$$\operatorname{vec}(A)_{i+m(j-1)}=A_{i,j},\quad\text{for}\quad 1\leq i\leq n,\enspace 1\leq j\leq m.$$ 
For simplicity, we will use the same symbols $\varrho$, $V_1$, $V_2$, $C$, $\mathcal{C}_k(U)_i$ and $\mathcal{L}_j(U)_i$ (for $i,j=1,\cdots,4$ and $k=1,2$) to denote both the original matrices and their vectorizations, whenever there is no risk of confusion.

Let $\otimes$ denotes the Kronecker product and $I_n$ the identity matrix of size $n$.
With this notation, the nonzero blocks of $\mathcal{C}_1$ can be expressed in vector form as
\begin{equation*}
        \mathcal{C}_1(U)_1 =
            (I_M\otimes D_M)(\varrho_{*,x}*V_1) + (D_M\otimes I_M)(\varrho_{*,y}*V_2)+ h\mathcal{A}_1\varrho,
\end{equation*}
with
$$\mathcal{A}_1(U)=(I_M\otimes D_M)\Lambda^x(I_M\otimes D_M^T) + (D_M\otimes I_M)\Lambda^y(D_M^T\otimes I_M),$$
where $\Lambda^x$ and $\Lambda^y$ denote the diagonal matrices of the maximum characteristic speeds associated with the fluxes $F^{\rho}$ and $G^{\rho}$ in the $x$- and $y$-directions, respectively, evaluated at the reconstructed states.
Similarly, the nonzero blocks of the operators $\mathcal{L}_1$, $\mathcal{L}_3$, $\mathcal{L}_4$ can be written as
\begin{equation*}
  \begin{split}
    \mathcal{L}_1(U)_2 &= (I_M\otimes D_M^T)p_2(\varrho),\\
    \mathcal{L}_1(U)_3 &= (A_M\otimes I_M)g\varrho + (D_M^T\otimes I_M)p_2(\varrho),\\
    \mathcal{L}_3(U)_4 &= 2\Delta_h C + \mathcal{A}_3(C)C - \Delta_h(D(\varrho^{-1})\Delta_h C),
  \end{split}
\end{equation*}
where $\mathcal{A}_3$ is the tensor constructed form the values of $\psi^{\prime\prime}_2$ in \eqref{psi_2_xx}-\eqref{psi_2_yy} and $\Delta_h$ the Laplacian operator in tensor form.
The nonzero blocks of $\mathcal{L}_4$ are the following:
\begin{equation*}
    \mathcal{L}_4(U)_2 = - A_{1,1} V_1 - A_{1,2} V_2,
    \quad
    \mathcal{L}_4(U)_3 = - A_{2,1} V_1 - A_{2,2} V_2,
\end{equation*}
where the matrices $A_{i,j}$ are given by
\begin{equation}\label{vel_matrices}
    \begin{split}
        A_{1,1} &=(2\nu+\lambda)I_M\otimes (D^T_MD_M)+\nu(D_{M+1}^T D^{\ast}_{M+1}) \otimes I_{M-1},\\
        A_{1,2} &=(\nu+\lambda)D_M \otimes D_M^T,\\
        A_{2,1} &=(\nu+\lambda)D_M^T \otimes D_M,\\
        A_{2,2} &=(2\nu+\lambda)(D_M^T D_M) \otimes I_M + \nu I_{M-1} \otimes (D_{M+1}^T D^{\ast}_{M+1}).\\
    \end{split}
\end{equation}
Using this notation, system \eqref{semi_discrete_system} can be expressed compactly in vector form by defining
\begin{equation*}
    U = 
    \begin{bmatrix}
        \varrho \\ \varrho_{*,x} * V_1 \\ \varrho_{*,y} * V_2 \\ \varrho * C
    \end{bmatrix},
    \quad
    \text{and}
    \quad
    \mathcal{L}(U) = 
    \mathcal{C}(U)
    + 
    \sum_{j=1}^4
    \begin{bmatrix} 
        \mathcal{L}_j(U)_1 \\ 
        \mathcal{L}_j(U)_2 \\ 
        \mathcal{L}_j(U)_3 \\ 
        \mathcal{L}_j(U)_4
    \end{bmatrix}.
\end{equation*}  
    
\subsection{Implicit-Explicit Schemes}\label{section_imex_schemes}

To construct an implicit-explicit scheme, we employ the technique of doubling variables combined with a partitioned Runge-Kutta approach \cite{BBMRV15,PR05,MMY25,mulet_24}.
Consider a sufficiently smooth function 
$$\tilde{\mathcal{L}}:\mathbb{R}^N\times\mathbb{R}^N\longrightarrow\mathbb{R},$$ 
defined as
$$\tilde{\mathcal{L}}(\tilde{U}, U) = \mathcal{C}_1(\tilde{U},U) + \mathcal{C}_2(\tilde{U}) + \tilde{\mathcal{L}}_1(\tilde{U},U) + \mathcal{L}_2(\tilde{U}) + \tilde{\mathcal{L}}_3(\tilde{U}, U) + \mathcal{L}_4(U),$$
where the only nonzero component of the operators $\tilde{\mathcal{L}}_1$ and $\tilde{\mathcal{L}}_3$ are given by
\begin{equation*}
    \begin{split}
        \tilde{\mathcal{C}}_1(\tilde{U}, U)_{1} &= (I_M\otimes D_M)(\varrho_{*,x}*V_1) + (D_M\otimes I_M)(\varrho_{*,y}*V_2)+ h\mathcal{A}_1(\tilde{U})\tilde{\varrho},\\
        \tilde{\mathcal{L}}_1(\tilde{U}, U)_{2} &= (I_M\otimes D_M^T)p_2(\varrho),\\
        \tilde{\mathcal{L}}_1(\tilde{U}, U)_{3} &= (A_M\otimes I_M)g\tilde{\varrho} + (D_M^T\otimes I_M)p_2(\varrho),\\
        \tilde{\mathcal{L}}_3(\tilde{U}, U)_{4} &= 2\Delta_h C + \mathcal{A}_3(\tilde{C})\tilde{C} - \Delta_h(D(\varrho^{-1})\Delta_h C).
    \end{split}
\end{equation*}
Using this operator, the full discrete scheme given by \eqref{semi_discrete_system} can be written as
\begin{equation}\label{discrete_ODE}
	\begin{split}
		& U^\prime = \tilde{\mathcal{L}}(U, U),\\
		& U(0) = U_0,
	\end{split}
\end{equation}
which is equivalent to 
\begin{equation}\label{partitioned_RK}
	\begin{split}
		\tilde{U}^\prime &= \tilde{\mathcal{L}}(\tilde{U}, U),\\
		U^\prime &= \tilde{\mathcal{L}}(\tilde{U}, U),\\
		\tilde{U}(0) &= U(0) = U_0.
	\end{split}
\end{equation}
Here, all terms involving $\tilde{U}$ are treated explicitly, while those depending on $U$ are handled implicitly.
System \eqref{partitioned_RK} allows us to apply separate Runge-Kutta schemes to the explicit and the implicit parts.
Therefore, we consider a pair of Butcher tableaus with $s$ stages:
\begin{equation*}
	\begin{array}{c|c}
		\tilde{\gamma} & \tilde{\alpha} \\
		\hline
		&\tilde{\beta}^T
	\end{array},
	\qquad
	\begin{array}{c|c}
		\gamma & \alpha \\
		\hline
		&\beta^T
	\end{array}.
\end{equation*}
The first tableau defines the explicit part of the scheme with $\tilde{\alpha}_{i,j} = 0$ for all $j\geq i$, while the second tableau represents the diagonally implicit part, where $\alpha_{i,j} = 0$ for $j>i$ and $\alpha_{i,i}\neq0$. 
The $\gamma_i$ and $\tilde{\gamma}_i$ coefficients are defined by 
$$\gamma_i = \sum_{j=1}^{i}\alpha_{i,j},\quad\text{and}\quad\tilde{\gamma}_i = \sum_{j=1}^{i-1}\tilde{\alpha}_{i,j}.$$
Using these tableaus, the stage values of the partitioned Runge-Kutta method applied to \eqref{partitioned_RK} are computed as follows:
\begin{equation*}
    \begin{split}
        \tilde{U}^{(i)} &= \tilde{U}^n + \Delta t\sum_{j<i}\tilde{\alpha}_{i, j}\tilde{\mathcal{L}}(\tilde{U}^{(j)}, U^{(j)}),\\[2pt]
		U^{(i)} &= U^n + \Delta t\sum_{j<i}\alpha_{i, j}\tilde{\mathcal{L}}(\tilde{U}^{(j)}, U^{(j)}) + \Delta t\alpha_{i, i}\tilde{\mathcal{L}}(\tilde{U}^{(i)}, U^{(i)}),\\[2pt]
        \tilde{U}^{n+1} &= \tilde{U}^n + \Delta t\sum_{j=1}^s\tilde{\beta}_{j}\tilde{\mathcal{L}}(\tilde{U}^{(j)}, U^{(j)}),\\[2pt]
		U^{n+1} &= U^n + \Delta t\sum_{j=1}^s\beta_{j}\tilde{\mathcal{L}}(\tilde{U}^{(j)}, U^{(j)}),
    \end{split}
\end{equation*}
If $\beta=\tilde{\beta}$ and $U^n=\tilde{U}^n$, then both solutions remain identical at every time step, which eliminates the need of doubling the number of variables.
In addition, as proven in \cite{PR05}, if both Butcher tableaus are second-order accurate, the resulting partitioned Runge-Kutta method is also second-order accurate.
Consequently, the final scheme is given by
\begin{equation}\label{PRK_scheme}
    \begin{split}
        \tilde{U}^{(i)} &= U^n + \Delta t\sum_{j<i}\tilde{\alpha}_{i, j}\tilde{\mathcal{L}}(\tilde{U}^{(j)}, U^{(j)}),\\[2pt]
        U^{(i)} &= U^n + \Delta t\sum_{j<i}\alpha_{i, j}\tilde{\mathcal{L}}(\tilde{U}^{(j)}, U^{(j)}) + \Delta t\alpha_{i, i}\tilde{\mathcal{L}}(\tilde{U}^{(i)}, U^{(i)}),\\[2pt]
        U^{n+1} &= U^n + \Delta t\sum_{j=1}^s\beta_{j}\tilde{\mathcal{L}}(\tilde{U}^{(j)}, U^{(j)}).
    \end{split}
\end{equation}

Henceforth, we restrict our analysis to Stiffly Accurate Runge-Kutta schemes, that is, those satisfying $\alpha_{s,j}=\beta_j$ for $j=1,\ldots,s$.

\subsection{Solution to the Nonlinear Systems}
At each intermediate stage $i=1,\cdots,s$, the scheme \eqref{PRK_scheme} reduces to solving the following nonlinear system for $U^{(i)}$:
\begin{equation}\label{nonlinear_eq}
    U^{(i)} = U^n + \Delta t\sum_{j<i}\alpha_{i, j}\tilde{\mathcal{L}}(\tilde{U}^{(j)}, U^{(j)}) + \Delta t\alpha_{i, i}\tilde{\mathcal{L}}(\tilde{U}^{(i)}, U^{(i)}),
\end{equation}
which is composed by two subsystems: a nonlinear system for the density and the velocities, and then a linear system for the $c$ variable, following the approach described in \cite{MMY25,mulet_24}.\\

For the nonlinear subsystem one has to solve $M^2 + 2M(M-1)$ equations for equal number of unknowns, corresponding to $\varrho$, $V_1$ and $V_2$. 
The system reads as: 
\begin{equation}\label{nonlinear_subsystem}
    \begin{aligned}
        \varrho^{(i)} - \widehat{\varrho^{(i)}} + 
            \Delta t\alpha_{i,i}\left((I_M\otimes D_M)(\varrho_{*,x}*V_1)^{(i)} + (D_M\otimes I_M)(\varrho_{*,y}*V_2)^{(i)}\right) &= 0,\\[8pt]
        (\varrho_{*,x}*V_1)^{(i)} - \widehat{(\varrho_{*,x}*V_1)^{(i)}} + \Delta t\alpha_{i,i}\left( A_{1,1}V_1^{(i)} + A_{1,2}V_2^{(i)} - (I_M\otimes D_M^T)p_2(\varrho^{(i)})\right) &= 0, \\[8pt]
        (\varrho_{*,y}*V_2)^{(i)} - \widehat{(\varrho_{*,y}*V_2)^{(i)}} + \Delta t\alpha_{i,i}\left( A_{2,1}V_1^{(i)} + A_{2,2}V_2^{(i)} - (D^T_M\otimes I_M)p_2(\varrho^{(i)})\right) &= 0, \\[8pt]
    \end{aligned}
\end{equation}
where the terms marked with a hat $\widehat{\cdot}$ are explicitly computed at the current stage $i=1,\cdots,s$. 

Once $\varrho^{(i)}$, $V_1^{(i)}$ and $V_2^{(i)}$ have been computed, the remaining step to solve in \eqref{nonlinear_eq} is the linear system for $C^{(i)}$.
The system takes the form,
\begin{equation*}
    (\varrho\ast C)^{(i)} -\widehat{(\varrho\ast C)^{(i)}} + \Delta t \alpha_{i,i}\left(\varepsilon\Delta_h\left(\frac{1}{\varrho^{(i)}}\Delta_h C^{(i)}\right) - 2\Delta_h C^{(i)}\right) = 0,
\end{equation*}
which is equivalent to solving the following linear system for $C^{(i)}$
\begin{equation}\label{system_c}
    \left(D(\varrho^{(i)}) - 2\Delta t\alpha_{i,i}\Delta_h + \Delta t\alpha_{i,i}\varepsilon\Delta_hD(\varrho^{(i)})^{-1}\Delta_h\right)C^{(i)} = \widehat{(\varrho\ast C)^{(i)}}.
\end{equation}
Due to the convex splitting stated in Section \ref{section_eyre_splitting}, the coefficient matrix is symmetric and positive definite, provided that $\varrho_{k}^{(i)}>0$ for all $k=1,\cdots,M^2$.

\subsection{Nonlinear Solvers}

For the nonlinear subsystem \eqref{nonlinear_subsystem} the damped Newton's method is employed \cite{BIMV19}.
Dropping the superscript of the $i$-stage, the nonlinear system \eqref{nonlinear_subsystem} expressed in compact form is
$$H(z)\equiv L(z)+ \Delta t\alpha_{i,i}\mathcal{D}(z) - r = 0,$$
where 
\begin{equation*}
    \mathbf{z} = 
    \begin{bmatrix}
        \varrho\\ V_1\\ V_2
    \end{bmatrix},
    \quad
    L(\mathbf{z}) = 
    \begin{bmatrix}
        \varrho\\ \varrho_{*,x}*V_1\\ \varrho_{*,y}*V_2
    \end{bmatrix},
    \quad
    r = 
    \begin{bmatrix}
        \widehat{\varrho}\\ \widehat{\varrho_{*,x}*V_1}\\ \widehat{\varrho_{*,y}*V_2}
    \end{bmatrix},
\end{equation*}
and the nonlinear operator is
\begin{equation*}
    \mathcal{D}(\mathbf{z}) =
    \begin{bmatrix}
        (I_M\otimes D_M)(\varrho_{*,x}*V_1) + (D_M\otimes I_M)(\varrho_{*,y}*V_2)\\
        A_{1,1}V_1 + A_{1,2}V_2 - (I_M\otimes D^T_M)p_2(\varrho)\\
        A_{2,1}V_1 + A_{2,2}V_2 - (D^T_M\otimes I_M)p_2(\varrho)
    \end{bmatrix}.
\end{equation*}
Then at each Newton iteration the solution is updated as 
$$z^{n+1} = z^n+ \alpha_n\delta^n,$$
where the step $\delta^n$ is computed by solving the linear system 
\begin{equation}\label{newton_linear_system}
  H^\prime(z^n)\delta^n = -H(z^n).
\end{equation}
Here, $H^\prime(z^n)$ denotes the Jacobian matrix of $H$ evaluated at the current iterate $z^n$, and the damping parameter $\alpha_n\in(0,1]$ is chosen to ensure that $||H(z^{n+1})||_2$ is decreasing.
The Jacobian matrix has the form
\begin{equation}\label{jacobian_matrix}
    H^\prime(\mathbf{z}) =
    L^\prime(z) + \Delta t\alpha_{i,i}\mathcal{D}^\prime(z) =
    \begin{bmatrix}
        \begin{array}{c|cc}
        I_{M^2} & 0_{M^2} \ 0_{M^2} \\ 
        \hline
        D_V     & D_\varrho
        \end{array}
    \end{bmatrix}
    + \Delta t\alpha_{i,i}
    \begin{bmatrix}
        \begin{array}{c|c}
        A_V &  A_\varrho\\
        \hline
        A_{p_2}
        & B
        \end{array}
    \end{bmatrix},
\end{equation}
where $0_{M^2}$ denotes the $M^2$ zero matrix,
\begin{align*}
    D_V &= 
    \begin{bmatrix}
        D(V_1)(I_M\otimes A_M)\\[5pt] D(V_2)(A_M\otimes I_M)
    \end{bmatrix},&
    \quad
    A_\varrho &=
    \begin{bmatrix}
        (I_M\otimes D_M)D(\varrho_{*,x}) \\[5pt] (D_M\otimes I_M)D(\varrho_{*,y})
    \end{bmatrix}^T, &
    \\[11pt]
    D_\varrho &= 
    \begin{bmatrix}
        D(\varrho_{*,x}) & 0_{M^2} \\[5pt]
        0_{M^2}      & D(\varrho_{*,y})
    \end{bmatrix},&
    \quad
    B &= 
    \begin{bmatrix}
        A_{1,1} & A_{1,2} \\[5pt]
        A_{2,1} & A_{2,2}
    \end{bmatrix},&
    \\[11pt]
    A_{p_2} &= -
    \begin{bmatrix}
        (I_M\otimes D^T_M)D(p^\prime_2)\\[5pt] (D^T_M\otimes I_M)D(p^\prime_2)
    \end{bmatrix},&
\end{align*}
and
$$A_V = (I_M\otimes D_M)D(V_1)(I_M\otimes A_M) + (D_M\otimes I_M)D(V_2)(A_M\otimes D_M).$$

For analyzing the invertibility of the Jacobian matrix \eqref{jacobian_matrix} the following result, proven in \cite{MMY25}, is needed.
\begin{proposition}
    If $\varrho_k>0$ for every $k=1,\ldots,M^2$, and $\nu,\lambda>0$, then $D_\varrho + \Delta t\alpha_{i,i}B$ is symmetric and strictly positive definite.
\end{proposition}

Consequently, assuming that $\varrho_k>0$ for every $k=1,\ldots,M^2$, the Jacobian matrix \eqref{jacobian_matrix} is invertible provided that 
\begin{equation*}
    \det\left(
        I_{M^2} + \Delta t\alpha_{i,i}\left(
            A_V - 
            A_\varrho(D_\varrho + \Delta t\alpha_{i,i}B)^{-1}(D_V+\Delta t\alpha_{i,i}A_{p_2})\right)
    \right) \neq 0.
\end{equation*}
Clearly, if $\Delta t\alpha_{i,i}$ were zero, the above condition holds.
Hence, for sufficiently small values of $\Delta t\alpha_{i,i}$ the Jacobian matrix $H^\prime$ is invertible.\\

In \cite{MMY25,mulet_24}, multigrid V-cycle algorithm with a small number of pre- and post- Gauss-Seidel smoothings was proven to be effective for solving system \eqref{system_c}.
In particular, this approach was successfully applied in \cite{MMY25,mulet_24} to the system formed by the sub-block of the Jacobian matrix \eqref{jacobian_matrix} given by
\begin{equation*}
    D_\varrho + \Delta t\alpha_{i,i}B.
\end{equation*}
The analysis of linear solvers for the complete Jacobian matrix \eqref{jacobian_matrix} for approximating the solution of system \eqref{newton_linear_system} is beyond the scope of this work.

\subsection{Time-Step selection}\label{section_time_step}

The time step is chosen based only on the convective part of the system.
It follows that the CFL stability condition for the proposed scheme takes the form
$$\Delta t = \text{CFL}^* \cdot \frac{\Delta x}{cs},$$ 
where $\text{CFL}^*$ is some constant less than one, and $cs$ denotes the maximum of characteristic speeds, computed as
\begin{equation}\label{time_step_cond}
	cs = \max \left\{ \left| V^{(i)}_{k,j} \right| + \sqrt{p^\prime_1(\varrho^{(i)}_j)} : i = 1,\cdots,s,\enspace k = 1, 2,\enspace j=1,\cdots,M^2\right\}.
\end{equation}
Since the stiff pressure component is treated implicitly, it does not influence the stability condition.
As a result, the time step $\Delta t$ is independent of the parameter $C_{p,2}$, depending only on the non-stiff pressure part $p_1$ and the velocity field $\mathbf{v}$.
This allows the method to avoid severe time-step restrictions typically faced on low Mach number regimes.

On the other hand, the proposed scheme is not guaranteed to be bound-preserving: the density can be negative or the order parameter $c$ can be outside the physical interval $[-1,1]$. 
As discussed in the monograph \cite{miranville}, the polynomial double-well potential \eqref{eq_poly_double_well} does not satisfy the maximum principle. 
Consequently, one cannot expect the discrete approximation of $c$ to remain strictly within $[-1,1]$.
Despite such limitation, this potential is widely used in the literature due to its simplicity \cite{Boyer99,DhaouadiDumbserGavrilyuk25,MMY25,mulet_24}, but 
it can be replaced by the classical logarithmic potentials described in \cite{AbelsFeireisl08,Boyer99,DhaouadiDumbserGavrilyuk25,miranville} and references therein.

Nevertheless, in our test the scheme preserves the positivity of the density and keeps $c$ almost within bounds, up to a small deviation.
Alternative techniques to mitigate these issues are presented in \cite{MMY25,mulet_24}, where the time step is reduced whenever $|c|$ exceeds a predefined threshold, and then it is gradually increased back.

\section{Asymptotic Preserving Property}\label{section_AP}

In \cite{abels_liu_necasova_24} it was proven that in the low Mach number limit of the compressible Cahn-Hilliard-Navier-Stokes system converges to its incompressible counterpart, under suitable initial conditions. 

The aim of this section, is to show that the proposed scheme is asymptotically stable. 
To formalize this notion, we recall the definition of an asymptotic preserving scheme provided in \cite{cordier_degond_kumbaro_12,degond_tang_11,Haack-Jin-Liu2012,NBALMM14}.  

\begin{definition}
    \em{Let \(\mathcal{M}^\delta\) be a continuous physical model depending on a perturbation parameter \(\delta\).  
    Define \(\mathcal{M}^0\) as the limiting model obtained from \(\mathcal{M}^\delta\) when \(\delta \to 0\).  
    A numerical scheme \(\mathcal{M}^\delta_\Delta\) for approximating \(\mathcal{M}^\delta\), where \(\Delta = (\Delta t, \Delta x)\) denotes the temporal and spatial discretization parameters, is said to be \textit{asymptotic preserving} (AP) if:  
    \begin{enumerate}
        \item its stability condition is independent of \(\delta\), and
        \item in the limit \(\delta \to 0\), the scheme \(\mathcal{M}^\delta_\Delta\) converges to a consistent discretization \(\mathcal{M}^0_\Delta\) of the continuous limiting model \(\mathcal{M}^0\).
    \end{enumerate}
    This concept is illustrated in Figure~\ref{fig_ap_diagram}.}
\end{definition}

\begin{figure}[ht]
    \centering
    \[
    \begin{tikzcd}[row sep=2.5em, column sep=3em]
        \mathcal{M}^\delta \arrow[r, "\delta \to 0"] & \mathcal{M}^0 \\
        \mathcal{M}_\Delta^\delta \arrow[u, "\Delta \to 0"'] \arrow[r, "\delta \to 0"'] & \mathcal{M}_\Delta^0 \arrow[u, "\Delta \to 0"']
    \end{tikzcd}
    \]
    \caption{
        Diagram illustrates the asymptotic-preserving (AP) property. 
        $\mathcal{M}^\delta$, $\mathcal{M}^0$ denotes the continuous compressible and incompressible system, while $\mathcal{M}_\Delta^\delta$, $\mathcal{M}_\Delta^0$ represents their discrete counterparts, respectively.
        The AP is verified if the diagram commutes.
    } 
    \label{fig_ap_diagram}
\end{figure}

We denote by $\mathcal{M}^\delta$ the compressible system \eqref{compressible_CHNS} in the low Mach number regime, and by $\mathcal{M}^\delta_\Delta$ its numerical discretization according to \eqref{PRK_scheme}. 
The corresponding incompressible Cahn-Hilliard-Navier-Stokes equations with gravitational acceleration is denoted by $\mathcal{M}^0$ and reads as follows:
\begin{equation}\label{incompressible_CHNS}
    \left\{
    \begin{split}
        &\operatorname{div}\mathbf{v} = 0,\quad \rho = \rho_0,\\
        &\mathbf{v}_t + \operatorname{div}(\mathbf{v}\otimes\mathbf{v}) + \nabla p_{(1)} 
            = \mathbf{g} + \frac{1}{\rho_0}\left(
                \nu\Delta\mathbf{v} + \operatorname{div}\mathbb{T}_2(c)
            \right),\\
        &c_t + \operatorname{div}(\mathbf{v}c) = \frac{1}{\rho_0}\Delta\mu,    
    \end{split}
    \right.
\end{equation}
where $\rho_0>0$ is the constant density of the incompressible mixture.
Here, $p_{(1)}$ denotes the scalar pressure, which acts as a Lagrange multiplier associated with the incompressibility constraint $\operatorname{div}\mathbf{v}=0$.
We denote by $\mathcal{M}^0_\Delta$ the discretization of \eqref{incompressible_CHNS} according to \eqref{PRK_scheme}.

Assume that the density, velocity field, concentration difference, and pressure admit the following expansions \cite{degond_tang_11,KM81,NBALMM14}:
\begin{equation}\label{hilbert_expansion}
    \begin{split}
        \rho(\mathbf{x},t) &= \rho_{(0)}(\mathbf{x},t) + \delta\rho_{(1)}(\mathbf{x},t) + \dots,\\[3pt]
        \mathbf{v}(\mathbf{x},t) &= \mathbf{v}_{(0)}(\mathbf{x},t) + \delta\mathbf{v}_{(1)}(\mathbf{x},t) + \dots,\\[3pt]
        c(\mathbf{x},t) &= c_{(0)}(\mathbf{x},t) + \delta c_{(1)}(\mathbf{x},t) + \dots,\\[3pt]
        p(\mathbf{x},t) &= p_{(0)}(\mathbf{x},t) + \delta p_{(1)}(\mathbf{x},t) + \dots,
    \end{split}
\end{equation}
and the well-preparedness of the data, that is, 
\begin{equation}\label{eq_well_prepared_data}
  \nabla\rho_{(0)} = 0,\quad \operatorname{div}\mathbf{v}_{(0)} = 0.
\end{equation}
Here, the terms in the pressure expansion follow directly from a Taylor series around $\rho_{(0)}$, so that $p_{(0)} = p\left(\rho_{(0)}\right)$, $p_{(1)} = p^\prime\left(\rho_{(0)}\right)\rho_{(1)}$, and higher-order terms are obtained similarly.\\

\begin{theorem}\label{teo_AP}
  Consider an IMEX Stiffly Accurate partitioned Runge-Kutta scheme with $\tilde{\beta} = \beta$ given by \eqref{PRK_scheme}.
  Assume that $U^{n}$ and each stage $\tilde{U}^{(l)}$, $U^{(l)}$ admit the decomposition \eqref{hilbert_expansion}.
  If $U^n$ verifies \eqref{eq_well_prepared_data}, then so does $U^{(l)}$.
  Furthermore, if $U^{n+1}$ admits the decomposition \eqref{hilbert_expansion}, then $U^{n+1}$ is well-prepared and the scheme is AP.
\end{theorem}
\begin{proof}
    For simplicity, in the proof, we shall assume that the pressure splitting stated in \eqref{pressure_splitting} is reformulated in terms of the Mach number, that is,
    $$\frac{1}{\delta} p(\rho) = C_{p,1}p(\rho) + \frac{1-C_{p,1}\delta}{\delta}p(\rho),\text{ for }p(\rho) = \rho^\gamma.$$
    We define $p^{(l)}=p\left(\rho^{(l)}\right)$ for each stage $l$.
    Let $\nabla_h$, $\operatorname{div}_h$ and $\Delta_h$ denote the discrete gradient, divergence and laplacian operators, and $\Lambda_{*}$ be the diagonal matrix of the numerical viscosities associated to the flux of the continuity equation.
    We prove the result by induction on the number of stages $s$.

    For one stage $s=1$. 
    First, let us show that $U^{(1)}$ is well-prepared. 
    We have that $\tilde{U}^{(1)} = U^n$, so the momentum part in $U^{(1)}$ is
    \begin{equation*}
      \begin{split}
        (\rho\mathbf{v})^{(1)} = (\rho\mathbf{v})^{n} 
        + \Delta t\alpha_{1,1}&\left[
          -\left(\operatorname{div}_{h}\left(\rho\mathbf{v}\otimes\mathbf{v} + C_{p,1}p(\rho)\mathbb{I}\right)\right)^{n}+
          \rho^{n}\mathbf{g} + 
          \operatorname{div}_{h}\left(\mathbb{T}_1v^{(1)}\right)\right.\\[5pt]
          &\left.\enspace +\operatorname{div}_{h}\left(\mathbb{T}_2c^{n}\right) -
          \frac{1-C_{p,1}\delta}{\delta}\nabla_{h} p\left(\rho^{(1)}\right)
        \right].
      \end{split}
    \end{equation*}
    Since $\alpha_{1,1}\neq0$, taking limits when $\delta\to0$ it is obtained that
    \begin{equation}\label{eq_grad_pressure_zero}
        \nabla_{h} p\left(\rho^{(1)}_{(0)}\right) = 0.
    \end{equation}
    By definition of the pressure yields that $\rho^{(1)}_{(0)}$ is constant.
    The leading terms in the implicit stage for the mass conservation equation are given by
    \begin{equation*}
        \begin{split}
            \rho^{(1)}_{(0)} 
            &= 
            \rho^{n}_{(0)} + \Delta t\alpha_{1,1}\left(\left(\operatorname{div}_h\left(\rho_{(0)}\mathbf{v}_{(0)}\right)\right)^{(1)} + \operatorname{div}_h\left(\tilde{\Lambda}_{*}\nabla_h\rho^{n}_{(0)}\right)\right)\\[5pt]
            &=
            \rho^{n}_{(0)} + \Delta t\alpha_{1,1}\rho_{(0)}^{(1)}\operatorname{div}_h\left(\mathbf{v}^{(1)}_{(0)}\right).
        \end{split}
    \end{equation*}
    Since both $\rho^{(1)}_{(0)}$ and $\rho^{n}_{(0)}$ are constant, summing up the above expression over all spatial indices, a telescope sum in the velocity terms appear, and the boundary contributions vanish due to the boundary conditions \eqref{bdry_cond}.
    Therefore, $\rho^{(1)}_{(0)}=\rho^{n}_{(0)}$. 
    Consequently, the divergence free condition is obtained from the mass conservation and that $\alpha_{i,i}\neq 0$, i.e.,
    \begin{equation}\label{eq_div_free_proof}
        \operatorname{div}_{h}\mathbf{v}^{(1)}_{(0)} = 0.
    \end{equation}
    Since the scheme is Stiffly Accurate, then $U^{n+1} = U^{(1)}$, so $U^{n+1}$ is well-prepared provided it admits the decomposition \eqref{hilbert_expansion}.

    Let us show that the scheme is AP for $s=1$.
    For simplicity, we shall assume that $\rho^{(1)}_{(0)} = 1$, otherwise all the terms must be scaled by $(\rho^{(1)}_{(0)})^{-1}$. 
    The leading terms in the momentum of $U^{(1)}$, are those involved in the $\mathcal{O}(1)$ and $\mathcal{O}(\delta)$ terms, 
    so applying that $\tilde{\rho}^{(1)}_{(0)}=\rho^n_{(0)}$, \eqref{eq_grad_pressure_zero} and \eqref{eq_div_free_proof}, yields 
    \begin{equation*}
      \begin{split}
        \mathbf{v}^{(1)}_{(0)} = \mathbf{v}^{n}_{(0)} + 
        \Delta t\alpha_{1,1}\left[
            \mathbf{g}
            -\left(\operatorname{div}_{h}\left({\mathbf{v}}_{(0)}\otimes{\mathbf{v}}_{(0)}\right)\right)^{n}
            + \operatorname{div}_{h}\mathbb{T}_2c^n_{(0)}
            + \nu\Delta_h\mathbf{v}^{(1)}_{(0)}
            - \nabla_{h} p_{(1)}^{(1)}
        \right]
      \end{split}.
    \end{equation*}
    Similarly, for the Cahn-Hilliard type equation the leading terms are the $\mathcal{O}(1)$, i.e.,
    \begin{equation*}
      c_{(0)}^{(1)} = c_{(0)}^{n} + 
      \Delta t\alpha_{1,1}\left[
        - \left(\operatorname{div}_{h}\left(\mathbf{v}_{(0)}c_{(0)}\right)\right)^n 
        + \Delta\left(
          \psi^\prime_{+}\left(c_{(0)}^{(1)}\right) + \psi^\prime_{-}\left(c_{(0)}^{n}\right) - \varepsilon\Delta c^{(1)}_{(0)}
        \right)
      \right].
    \end{equation*}
    Since the scheme is Stiffly Accurate, then $U^{n+1} = U^{(1)}$, so the scheme is AP for $s=1$.\\

    We assume that the result is true for the first $s-1$ stages, and we prove it for stage $s$.    
    The momentum equation in $U^{(s)}$ is given by
    \begin{equation*}
      \begin{split}
        \left(\rho\mathbf{v}\right)^{(s)} 
        =& \left(\rho\mathbf{v}\right)^{n} +
        \Delta t\sum_{j\leq s}\alpha_{s,j}\left(
          -\left(\operatorname{div}_{h}\widetilde{\left(\rho\mathbf{v}\otimes\mathbf{v} + C_{p,1}p\mathbb{I}\right)}\right)^{(j)} + 
          \tilde{\rho}^{(j)}\mathbf{g} +
          \operatorname{div}_{h}\mathbb{T}_1\mathbf{v}^{(j)}\right.\\[5pt]
          &\left. + \operatorname{div}_{h}\mathbb{T}_2\tilde{c}^{(j)}_j
          -\frac{1-C_{p,1}\delta}{\delta}\nabla_{h} p\left(\rho^{(j)}\right)
        \right).
      \end{split}
    \end{equation*}
    Taking the limit when $\delta\to0$ with $\alpha_{s,s}\neq0$ and applying that the first $s-1$ implicit stages are well-prepared, yields that
    $$\nabla_{h} p\left(\rho^{(s)}_{(0)}\right) = 0,$$
    so $\rho^{(s)}_{(0)}$ is spatially constant.
    
    It follows inductively, that the leading terms in the conservation of mass in $\tilde{U}^{(l)}$ for $l=1,\cdots,s$ are
    \begin{equation}\label{eq_well_prepared_rho_tilde}
        \tilde{\rho}^{(l)}_{(0)} 
      = \rho^{n}_{(0)} - 
      \Delta t\sum_{j<l}\tilde{\alpha}_{i,j}\left(
        \operatorname{div}_{h}\left(\left(\rho_{(0)}\mathbf{v}_{(0)}\right)^{(j)}\right) + 
        \operatorname{div}_h\left(\tilde{\Lambda}_{*}\nabla_h\tilde{\rho}^{(j)}_{(0)}\right)\right) 
    = \rho^n_{(0)},
    \end{equation}
    since $U^{(l)}$ is well-prepared for $l=0,\cdots,s-1$ and $\rho^n_{(0)}$ is constant. 
    Therefore, $\tilde{\rho}^{(l)}_{(0)}$ is constant for $l=1,\cdots,s$. 
    Applying again the induction hypothesis and \eqref{eq_well_prepared_rho_tilde}, the leading terms in the conservation of mass of the implicit $s$-stage are:
    \begin{equation*}
        \begin{split}
            \rho^{(s)}_{(0)} 
            &= \rho^n_{(0)} - \Delta t\sum_{j\leq s}\alpha_{s,j}\left(\operatorname{div}_{h}\left(\rho\mathbf{v}\right)^{(j)}_{(0)} + \operatorname{div}\left(\tilde{\Lambda}_{*}\nabla_h\tilde{\rho}^{(j)}_{(0)}\right)\right)\\[5pt]
            &= \rho^n_{(0)} - \Delta t\sum_{j\leq s}\alpha_{s,j}\rho^{(j)}_{(0)}\operatorname{div}_{h}\left(\mathbf{v}^{(j)}_{(0)}\right)\\[5pt]
            &= \rho^n_{(0)} - \rho^{(s)}_{(0)}\Delta t\alpha_{s,s}\operatorname{div}_{h}\mathbf{v}^{(s)}_{(0)}.
        \end{split}
    \end{equation*}
    Proceeding as before, adding all terms up in the previous expression yields that $\rho^{(s)}_{(0)}=\rho^{n}_{(0)}$, which implies that 
    $$\operatorname{div}_{h}\mathbf{v}^{(s)}_{(0)} = 0.$$
    So $U^{n+1}$ is well-prepared since $\alpha_{s,j}=\beta_j$ for every $j=1,\cdots, s$.

    It only remains to show that the scheme of $s$-stages is AP. 
    Similarly, we assume that $\rho^{(l)}_{(0)} = 1$ for $l=1,\ldots,s$. 
    Applying that $U^{(l)}$ is well-prepared for $l=1,\ldots,s$ and \eqref{eq_well_prepared_rho_tilde}, then the leading terms of the momentum equation at each $l$-stage are given by
    \begin{equation*}
      \begin{split}
        &\tilde{\mathbf{v}}_{(0)}^{(l)} 
        = 
        \mathbf{v}_{(0)}^{n} +
        \Delta t\sum_{j< l}\tilde{\alpha}_{l,j}\left(
          \mathbf{g} 
          - \operatorname{div}_{h}\widetilde{\left(\mathbf{v}\otimes\mathbf{v}\right)}_{(0)}^{(j)}
          + \nu\Delta_h\mathbf{v}_{(0)}^{(j)}
          + \operatorname{div}_{h}\mathbb{T}_2\tilde{c}_{(0)}^{(j)}
          - \nabla_{h} p_{(1)}^{(j)}
        \right),\\[5pt]
        &\mathbf{v}_{(0)}^{(l)} 
        = 
        \mathbf{v}_{(0)}^{n} +
        \Delta t\sum_{j\leq l}\alpha_{l,j}\left(
          \mathbf{g} 
          - \operatorname{div}_{h}\widetilde{\left(\mathbf{v}\otimes\mathbf{v}\right)}_{(0)}^{(j)}
          + \nu\Delta_h\mathbf{v}_{(0)}^{(j)}
          + \operatorname{div}_{h}\mathbb{T}_2\tilde{c}_{(0)}^{(j)}
          - \nabla_{h} p_{(1)}^{(j)}
        \right).
      \end{split}
    \end{equation*}
    For the Cahn-Hilliard type equation it is obvious. 
    The AP property follows from the fact that the scheme is Stiffly Accurate.
\end{proof}

\section{Numerical experiments}\label{section_numerical_experiments}
The numerical experiments are presented in this section. 
The main objectives are as follows:
\begin{enumerate}
  \item To show that the order of the global convergence error agrees with the order of the numerical scheme.
  \item To verify that the number of time steps required by the IMEX scheme is consistent with the stability restriction imposed by the convective subsystem \eqref{time_step_cond}.
  \item To explain the properties preserved by the scheme, such as mass conservation, region preservation along with the CFL value (see Section \ref{section_time_step}), and others. 
\end{enumerate}
In all our experiments, the initial CFL number is set to $0.4$, the adiabatic exponent $\gamma$ fixed to $\frac{5}{3}$ and the parameters are set to 
$$\nu=1, \quad \lambda=10^{-1},\quad \varepsilon=10^{-4},\quad g=-10.$$
We define $C_{p,1}=\sqrt{C_p}$ and $C_{p,2}=C_{p} - C_{p,1}$. 
This choice has proven to be effective, as the experiments conducted under this setting have been successful.

All experiments were performed using a MATLAB R2024a implementation on a Linux machine running on 32 core of an AMD EPYC 7282.

We consider Stiffly Accurate Runge-Kutta schemes. 
In particular, we use a first-order method defined by the following Butcher tableau:
\begin{equation*}
  \text{EE-IE}
  \begin{array}{c|c}
    0 & 0  \\
    \hline
        & 1 \\
    \end{array},
    \qquad
    \begin{array}{c|c}
    1  &  1  \\
    \hline
        &  1 \\
    \end{array},
\end{equation*}
and a second-order method given by the $^\ast$-DIRKSA scheme:
\begin{equation*}
  ^\ast\text{-DIRKSA}
    \begin{array}{c|cc}
    0 & 0 & 0 \\
    1 + s & 1 + s & 0 \\
    \hline
        & s & 1 - s \\
    \end{array},
    \qquad
    \begin{array}{c|cc}
    1 - s & 1 -s & 0  \\
    1 &  {s} & {1-s}  \\
    \hline
        &  {s} & {1-s} \\
    \end{array},
    \quad s = \frac{1}{\sqrt{2}}.
\end{equation*}

\subsection{Order Tests}
In this section we show that the $^*$-DIRKSA scheme attains second-order of convergence. 
To this end, we introduce a forcing term into the equations, ensuring that the solution follows a prescribed analytical form. 
Specifically, the exact solution for the one-dimensional \eqref{eq_compressible_chns_1D} case is defined as
\begin{align*}
  \rho^\ast(x, t) &= 1 + \delta \cos(2\pi x) (t + 1), \\
  v^\ast_1(x, t) &= 0, \\
  c^\ast(x, t)   &= \frac{3}{4} - 0.1 (1 - \delta) \cos(\pi x) (t - 1),
\end{align*}
and for the two-dimensional case \eqref{eq_compressible_chns_2D},
\begin{align*}
  \rho^\ast(x, y, t) &= 1 + \delta \cos(2\pi x) \cos(\pi y) (t + 1), \\
  v^\ast_1(x, y, t) &= (1 + \delta) \big(1 - \cos(2\pi x)\big) \sin(2\pi y) (1 - 2t^2), \\
  v^\ast_2(x, y, t) &= (1 + \delta) \big(1 - \cos(2\pi y)\big) \sin(2\pi x) (2t^2 - 1), \\
  c^\ast(x, y, t)   &= \frac{3}{4} - 0.1 (1 - \delta) \cos(\pi x) \cos(\pi y) (t - 1).
\end{align*}
Notice that in both cases the initial velocity filed is divergence-free and the density is constant in space at $\mathcal{O}(\delta)$.

For the performance, the squared Mach numbers are taken as $\delta = 10^{-k}$ for $k=1,\ldots,8$.
The time-step $\Delta t$ is determined by the convective subsystem according to \eqref{time_step_cond}.
We consider meshes of size $M = 2^{i}$ for $i = 3,\cdots,8$. 
The global errors and the experimental orders of convergence (EOC) are evaluated at $T= 0.01$ and are computed as
\begin{equation*}
  e_{M} = h^2\sum_{k=1}^{4}\!\sum_{i,j}^M\!|u^n_{k,i,j} - u_{k}(\mathbf{x}_{i,j}, T)|,
  \quad
  \text{EOC}_M = \log_2\left(\frac{e_{M}}{e_{2M}}\right).
\end{equation*}
The results for the one- and two-dimensional cases are shown in Table \ref{table_orders_convergence}.
In both tables, the $^\ast$-DIRKSA scheme achieves second-order convergence, while the EE-IE scheme is first-order in the one-dimensional case and in the two-dimensional case for $C_p\leq10^6$ decreases away from two.
When $C_p=10^7,10^8$ in the latter case, the order of convergence seems to tend to two.
One possible explanation is that the value of $C_{p,1}$ increases significantly as $C_p$ increases, making the time-integrator more robust.
Consequently, the spatial discretization order dominates the convergence.
Nevertheless, this phenomenon also occurs for $C_p\leq10^6$ until $M$ becomes sufficiently large, as illustrated in Table \ref{table_orders_convergence}.
We expect that for $M > 1024$ the mentioned orders tends to 1, but due to the computational cost, we did not perform such experiments.
In the remainder of this work, we restrict our experiments to the $^\ast$-DIRKSA scheme, since it consistently attains second-order of convergence.

\begin{table}
    \centering
    \setlength{\tabcolsep}{3pt}
    \begin{tabular}{cccccc|cccc}
    &   & \multicolumn{4}{c}{1D}                      & \multicolumn{4}{c}{2D}                      \\
    \hline
    &   & \multicolumn{2}{c}{$^\ast$-\text{DIRKSA}} & \multicolumn{2}{c}{\text{EE-IE}} & \multicolumn{2}{c}{$^\ast$-\text{DIRKSA}} & \multicolumn{2}{c}{\text{EE-IE}} \\
    \hline
        $C_p$ & $M$ & $e_M$ & EOC$_M$ & $e_M$ & EOC$_M$ & $e_M$ & EOC$_M$ & $e_M$ & EOC$_M$ \\
        \hline
        \multirow{6}{*}{$10$}
        & 8     & 1.317e-03 & ---   & 8.799e-04 & ---   & 2.4179e-02 & ---  & 2.0659e-02 & ---  \\
        & 16    & 2.608e-04 & 2.336 & 2.113e-04 & 2.058 & 6.8859e-03 & 1.81 & 6.3482e-03 & 1.70 \\
        & 32    & 6.411e-05 & 2.024 & 1.490e-04 & 0.504 & 1.8061e-03 & 1.93 & 1.8142e-03 & 1.81 \\
        & 64    & 1.609e-05 & 1.995 & 8.938e-05 & 0.738 & 4.5478e-04 & 1.99 & 5.0700e-04 & 1.84 \\
        & 128   & 4.028e-06 & 1.998 & 4.947e-05 & 0.853 & 1.1369e-04 & 2.00 & 1.4907e-04 & 1.77 \\
        & 256   & 1.010e-06 & 1.996 & 2.665e-05 & 0.893 & 2.8398e-05 & 2.00 & 4.8390e-05 & 1.62 \\
        & 512   & 2.525e-07 & 2.000 & 1.365e-05 & 0.965 & 7.0908e-06 & 1.86 & 1.7788e-05 & 1.19 \\
        & 1024  & 6.317e-08 & 1.999 & 6.933e-06 & 0.977 & 1.76861-06 & 2.00 & 7.3560e-06 & 1.27 \\
        \hline
        \multirow{6}{*}{$10^2$} 
        & 8     & 1.216e-03 & ---   & 6.669e-04 & ---   & 1.9006e-02 & ---  & 1.6018e-02 & ---  \\
        & 16    & 2.904e-04 & 2.066 & 1.247e-04 & 2.419 & 6.0020e-03 & 1.66 & 5.6030e-03 & 1.52 \\
        & 32    & 7.263e-05 & 1.999 & 7.355e-05 & 0.762 & 1.6209e-03 & 1.89 & 1.6349e-03 & 1.78 \\
        & 64    & 1.814e-05 & 2.001 & 5.459e-05 & 0.430 & 4.1116e-04 & 1.98 & 4.5097e-04 & 1.86 \\
        & 128   & 4.540e-06 & 1.998 & 3.312e-05 & 0.721 & 1.0310e-04 & 2.00 & 1.2913e-04 & 1.80 \\
        & 256   & 1.135e-06 & 2.000 & 1.772e-05 & 0.902 & 2.5790e-05 & 2.00 & 4.0565e-05 & 1.67 \\
        & 512   & 2.837e-07 & 2.000 & 9.175e-06 & 0.949 & 6.4443e-06 & 2.00 & 1.4410e-05 & 1.49 \\
        & 1024  & 7.094e-08 & 2.000 & 4.686e-06 & 0.969 & 1.6087e-06 & 2.00 & 5.8013e-06 & 1.31 \\
        \hline
        \multirow{6}{*}{$10^3$}
        & 8     & 6.203e-04 & ---   & 3.995e-04 & ---   & 1.2832e-02 & ---  & 1.1383e-02 & ---  \\
        & 16    & 1.246e-04 & 2.316 & 9.067e-05 & 2.139 & 5.5906e-03 & 1.20 & 5.2857e-03 & 1.11 \\
        & 32    & 2.866e-05 & 2.120 & 2.532e-05 & 1.840 & 1.6172e-03 & 1.79 & 1.6096e-03 & 1.72 \\
        & 64    & 7.024e-06 & 2.029 & 1.422e-05 & 0.832 & 4.1831e-04 & 1.95 & 4.4163e-04 & 1.87 \\
        & 128   & 1.749e-06 & 2.006 & 8.406e-06 & 0.759 & 1.0544e-04 & 1.99 & 1.2178e-04 & 1.86 \\
        & 256   & 4.369e-07 & 2.002 & 4.508e-06 & 0.899 & 2.6408e-05 & 2.00 & 3.5516e-05 & 1.78 \\
        & 512   & 1.092e-07 & 2.000 & 2.367e-06 & 0.930 & 6.6054e-06 & 2.00 & 1.1369e-05 &1.64  \\
        & 1024  & 2.730e-08 & 2.000 & 1.203e-06 & 0.977 & 1.6524e-06 & 2.00 & 4.0807e-06 &1.48  \\
        \hline
        \multirow{6}{*}{$10^4$}
        & 8     & 1.312e-04 & ---   & 1.258e-04 & ---   & 6.5014e-03 & ---  & 6.1598e-03 & ---  \\
        & 16    & 6.469e-05 & 1.020 & 4.345e-05 & 1.534 & 4.7449e-03 & 0.45 & 4.6011e-03 & 0.42 \\
        & 32    & 1.728e-05 & 1.905 & 1.717e-05 & 1.340 & 1.5540e-03 & 1.61 & 1.5527e-03 & 1.57 \\
        & 64    & 4.273e-06 & 2.015 & 6.935e-06 & 1.308 & 4.1191e-04 & 1.92 & 4.2678e-04 & 1.86 \\
        & 128   & 1.063e-06 & 2.007 & 3.404e-06 & 1.027 & 1.0445e-04 & 1.98 & 1.1448e-04 & 1.90 \\
        & 256   & 2.653e-07 & 2.002 & 1.840e-06 & 0.888 & 2.6203e-05 & 1.99 & 3.1766e-05 & 1.85 \\
        & 512   & 6.630e-08 & 2.001 & 9.787e-07 & 0.911 & 6.5566e-06 & 2.00 & 9.4616e-06 & 1.75 \\
        & 1024  & 1.657e-08 & 2.000 & 5.067e-07 & 0.950 & 1.6400e-06 & 2.00 & 3.1202e-06 & 1.60 \\
        \hline
        \multirow{6}{*}{$10^5$}
        & 8     & 1.338e-04 & ---   & 1.152e-04 & ---   & 2.1418e-02 & ---  & 2.0674e-02 & ---  \\
        & 16    & 3.780e-05 & 1.824 & 3.461e-05 & 1.736 & 3.8763e-03 & 2.47 & 3.7918e-03 & 2.45 \\
        & 32    & 1.021e-05 & 1.888 & 1.175e-05 & 1.558 & 1.4648e-03 & 1.40 & 1.4683e-03 & 1.37 \\
        & 64    & 2.013e-06 & 2.343 & 4.435e-06 & 1.406 & 4.0482e-04 & 1.86 & 4.1362e-04 & 1.83 \\
        & 128   & 5.653e-07 & 1.832 & 1.907e-06 & 1.218 & 1.0380e-04 & 1.96 & 1.0967e-04 & 1.92 \\
        & 256   & 1.451e-07 & 1.962 & 8.772e-07 & 1.120 & 2.6128e-05 & 1.99 & 2.9357e-05 & 1.90 \\
        & 512   & 3.650e-08 & 1.991 & 4.304e-07 & 1.027 & 6.5437e-06 & 1.99 & 8.2499e-06 & 1.83 \\
        & 1024  & 9.139e-09 & 1.998 & 2.162e-07 & 0.994 & 1.6367e-06 & 2.00 & 2.5131e-06 & 1.71 \\
        \hline
        \multirow{6}{*}{$10^6$}
        & 8     & 1.295e-04 & ---   & 1.189e-04 & ---   & 5.4800e-02 & ---  & 5.3752e-02 & ---  \\
        & 16    & 3.100e-05 & 2.062 & 3.175e-05 & 1.905 & 3.7253e-03 & 3.88 & 3.6616e-03 & 3.88 \\
        & 32    & 8.049e-06 & 1.946 & 9.094e-06 & 1.804 & 1.3844e-03 & 1.43 & 1.3855e-03 & 1.40 \\
        & 64    & 2.135e-06 & 1.915 & 3.101e-06 & 1.552 & 3.9548e-04 & 1.81 & 4.0041e-04 & 1.79 \\
        & 128   & 5.455e-07 & 1.968 & 1.200e-06 & 1.370 & 1.0315e-04 & 1.94 & 1.0638e-04 & 1.91 \\
        & 256   & 1.345e-07 & 2.020 & 5.229e-07 & 1.198 & 2.6102e-05 & 1.98 & 2.7888e-05 & 1.93 \\
        & 512   & 3.307e-08 & 2.024 & 2.455e-07 & 1.091 & 6.5447e-06 & 1.99 & 7.5033e-06 & 1.83 \\
        & 1024  & 8.226e-09 & 2.007 & 1.194e-07 & 1.040 & 1.6374e-06 & 2.00 & 2.1390e-06 & 1.81 \\
        \hline
        \multirow{6}{*}{$10^7$}
        & 8     & 1.691e-04 & ---   & 1.596e-04 & ---   & 1.1107e-01 & ---  & 1.1000e-01 & ---  \\
        & 16    & 3.108e-05 & 2.444 & 3.056e-05 & 2.384 & 6.1995e-03 & 4.16 & 6.1436e-03 & 4.16 \\
        & 32    & 7.937e-06 & 1.969 & 8.362e-06 & 1.870 & 1.3430e-03 & 2.21 & 1.3431e-03 & 2.19 \\
        & 64    & 1.991e-06 & 1.995 & 2.387e-06 & 1.809 & 3.8404e-04 & 1.81 & 3.8672e-04 & 1.80 \\
        & 128   & 4.974e-07 & 2.001 & 8.252e-07 & 1.532 & 1.0202e-04 & 1.91 & 1.0379e-04 & 1.90 \\
        & 256   & 1.279e-07 & 1.960 & 3.265e-07 & 1.338 & 2.6019e-05 & 1.97 & 2.7000e-05 & 1.94 \\
        & 512   & 3.242e-08 & 1.980 & 1.447e-07 & 1.174 & 6.5401e-06 & 1.99 & 6.9520e-06 & 1.96 \\
        & 1024  & 8.413e-09 & 1.946 & 6.849e-08 & 1.079 & ---        & ---  & ---        & ---  \\
        \hline
        \multirow{6}{*}{$10^8$}
        & 8     & 2.433e-04 & ---   & 2.374e-04 & ---   & 2.0294e-01 & ---  & 2.0219e-01 & ---  \\
        & 16    & 3.118e-05 & 2.964 & 3.076e-05 & 2.948 & 1.4330e-02 & 3.82 & 1.4267e-02 & 3.82 \\
        & 32    & 7.918e-06 & 1.977 & 8.071e-06 & 1.930 & 1.4663e-03 & 3.29 & 1.4650e-03 & 3.28 \\
        & 64    & 1.989e-06 & 1.993 & 2.134e-06 & 1.919 & 3.7428e-04 & 1.97 & 3.7570e-04 & 1.96 \\
        & 128   & 4.978e-07 & 1.998 & 6.250e-07 & 1.772 & 1.0034e-04 & 1.90 & 1.0132e-04 & 1.89 \\
        & 256   & 1.253e-07 & 1.990 & 2.207e-07 & 1.502 & 2.5881e-05 & 1.95 & 2.6416e-05 & 1.94 \\
        & 512   & 3.168e-08 & 1.984 & 8.928e-08 & 1.305 & 6.5312e-06 & 1.99 & 6.8161e-06 & 1.96 \\
        & 1024  & 8.156e-09 & 1.957 & 4.014e-08 & 1.153 & ---        & ---  & ---        & ---  \\
        \hline
    \end{tabular}
    \caption{
        $L_1$ errors and experimental order of convergence for the DIRKSA and EE-IE IMEX schemes for both the one- and two-dimensional case, 
        evaluated for different $C_p$ values with $C_{p,1} = \sqrt{C_p}$, in the test using a forced solution.}
    \label{table_orders_convergence}
\end{table}

\subsection{Test 1, 2 and 3}

In this section, we evaluate the performance of the following tests for several stiff pressure coefficients $C_p=10^{2k}$ for $k=1,\ldots,4$.
For this purpose, mass conservation, region preserving for the $c$-variable and the limit properties of the compressible scheme are discussed numerically.\\

\textbf{Test 1} 
This test is designed to show that the method remains stable even when the initial condition for the $c$-variable lies within the unstable region $(-\frac{1}{\sqrt 3}, \frac{1}{\sqrt 3})$ (see \cite{Elliott89,MMY25,mulet_24}).
In particular, we consider the following initial conditions:
\begin{align*}
  &\rho_0(x, y)=1 + \delta\cos(2 \pi x) \cos(\pi y),\\
  &\mathbf{v}_0(x, y) = (1 + \delta)\left(
    (1 - \cos(2\pi x)) \sin(2\pi y), (\cos(2\pi y) - 1) \sin(2\pi x)
  \right), \\
  &c_0(x, y)= 0.1(1 - \delta)\cos(\pi x)\cos(\pi y),
\end{align*}
which clearly verify the boundary conditions \eqref{bdry_cond} and the divergence-free condition for the velocity field.
The density is constant in space at leading orders of the low Mach number.

It is observed in Figures \ref{fig_test_1_Cp_10_8_1}, \ref{fig_test_1_Cp_10_8_2}, \ref{fig_test_1_Cp_10_8_3} that initially the density is dispersed and the order parameter $c$ lies within the interval $(-\frac{1}{\sqrt 3}, \frac{1}{\sqrt 3})$.
However, as the simulation evolves, the density is gradually increased near the bottom boundary $y=0$ due to the effects of gravitation.
In addition, the evolution of the $c$-variable illustrates the process of phase separation where complex patterns are formed.\\

\begin{figure}[h!]
    \centering
    \begin{minipage}[b]{0.45\textwidth}
        \centering
        {\small $T=0$} \\[0.3em]
        \includegraphics[width=\textwidth]{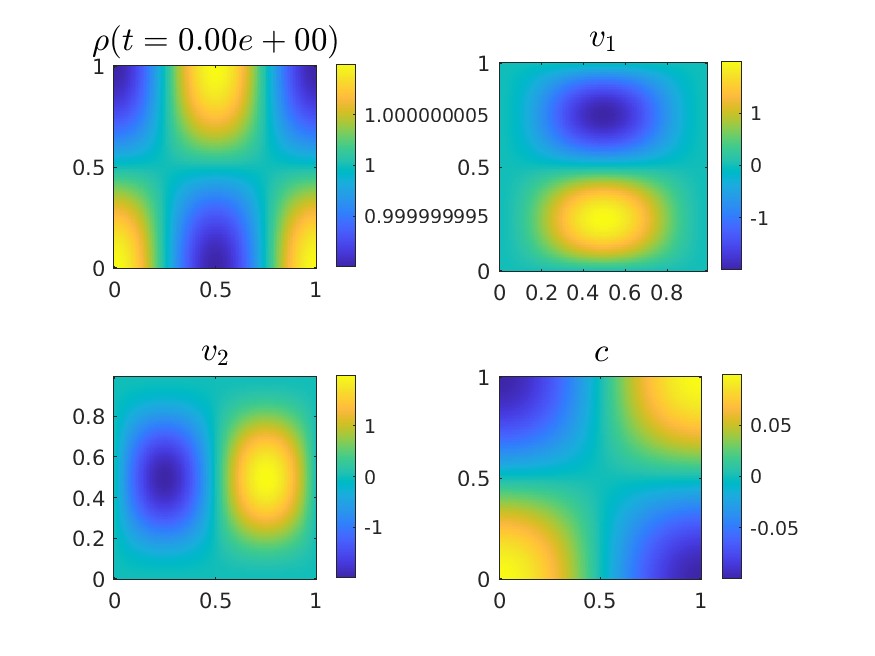}
    \end{minipage}
    \hfill
    \begin{minipage}[b]{0.45\textwidth}
        \centering
        {\small $T=0.01$} \\[0.3em]
        \includegraphics[width=\textwidth]{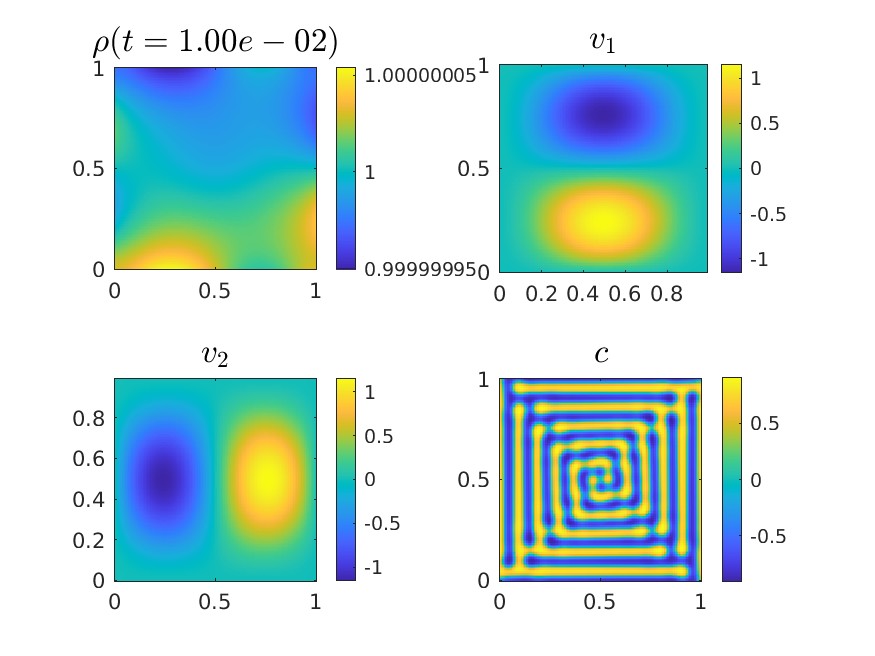}
    \end{minipage}
    \caption{Results for Test 1, $T=0, 0.01$, $M=128$ and $C_p=10^8$. 
    Initially, $c$ is lies within the unstable region. 
    At the beginning of the simulation, phase separation occurs. 
    Moreover, the density starts to become higher in the lower part of the domain due to gravity.}
    \label{fig_test_1_Cp_10_8_1}
\end{figure}

\begin{figure}[h!]
    \centering
    \begin{minipage}[b]{0.45\textwidth}
        \centering
        {\small $T=0.03$} \\[0.3em]
        \includegraphics[width=\textwidth]{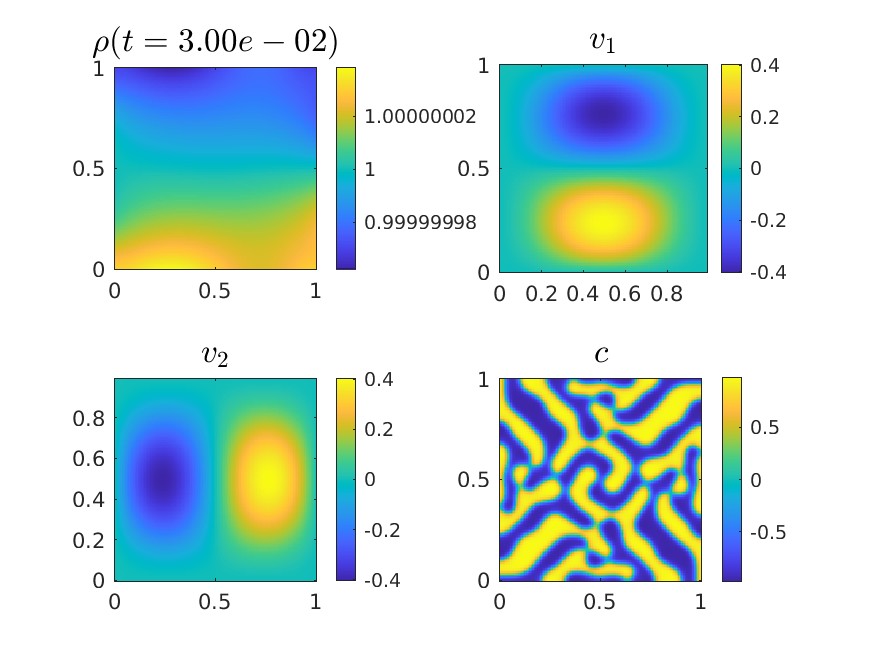}
    \end{minipage}
    \hfill
    \begin{minipage}[b]{0.45\textwidth}
        \centering
        {\small $T=0.05$} \\[0.3em]
        \includegraphics[width=\textwidth]{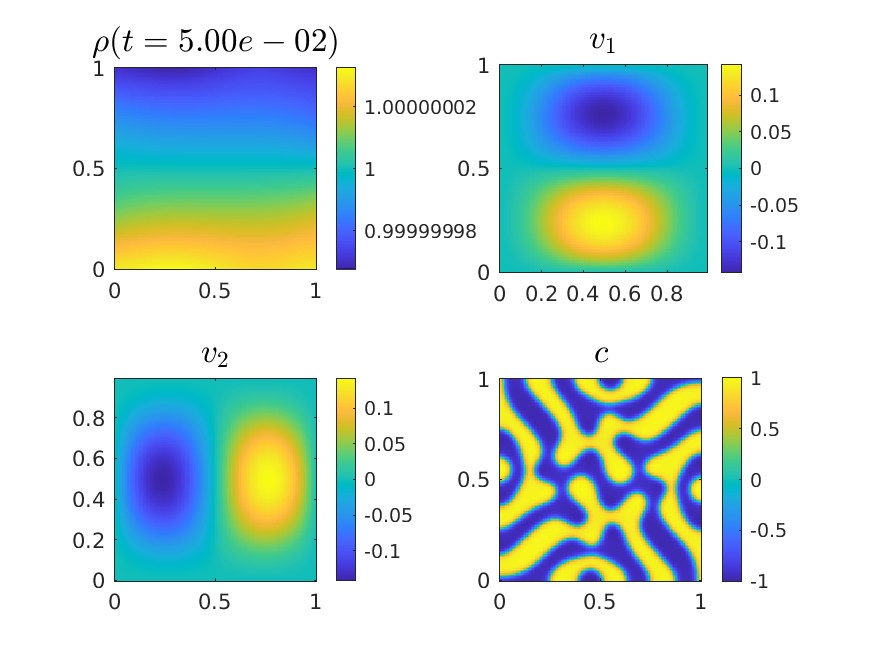}
    \end{minipage}
    \caption{Results for Test 1, $T=0.03, 0.05$, $M=128$ and $C_p=10^8$. 
    The process of spinodal decomposition continues, and the density is accumulating at the bottom of the domain due to gravity.}
    \label{fig_test_1_Cp_10_8_2}
\end{figure}

\begin{figure}[h!]
    \centering
    \begin{minipage}[b]{0.45\textwidth}
        \centering
        {\small $T=0.07$} \\[0.3em]
        \includegraphics[width=\textwidth]{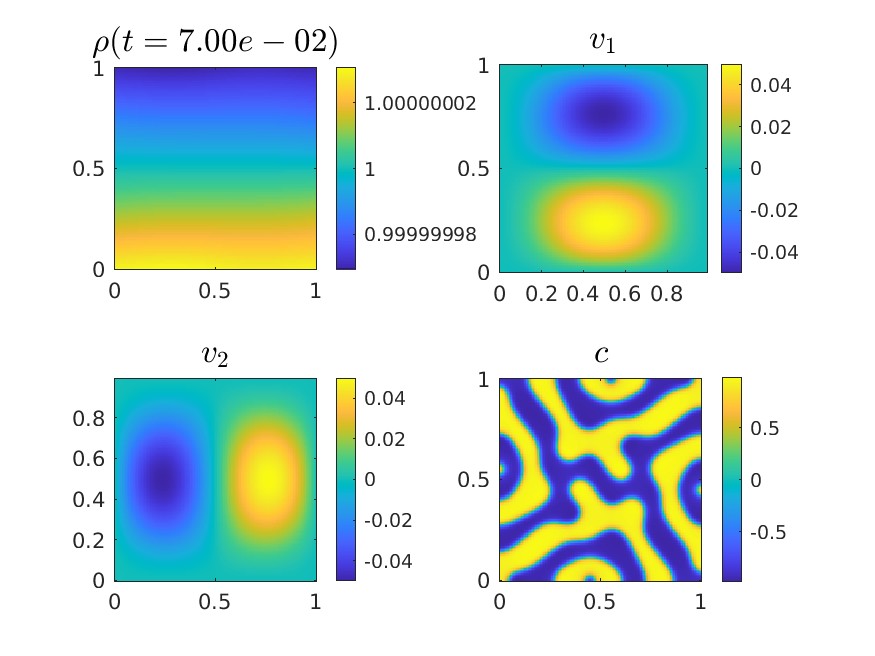}
    \end{minipage}
    \hfill
    \begin{minipage}[b]{0.45\textwidth}
        \centering
        {\small $T=0.1$} \\[0.3em]
        \includegraphics[width=\textwidth]{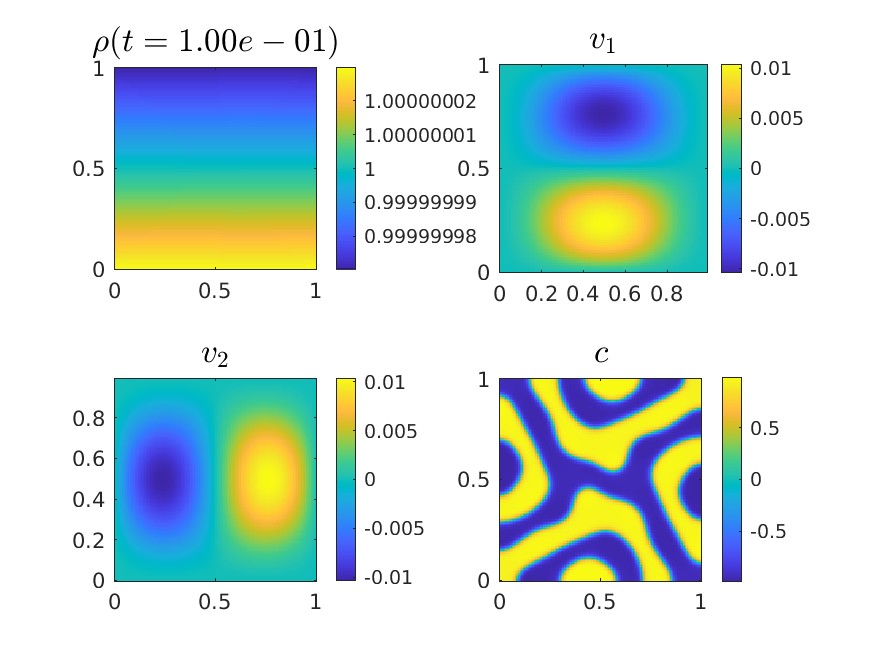}
    \end{minipage}
    \caption{Results for Test 1, $T=0.07, 0.1$, $M=128$ and $C_p=10^8$. 
    It can be observed that density remains almost constant among distinct times and that phase separation has almost finished.}
    \label{fig_test_1_Cp_10_8_3}
\end{figure}

\textbf{Test 2} 
The objective of this test is to assess the performance of the scheme when the order parameter $c$ initially lies outside the unstable region.
We consider:
\begin{align*}
  &\rho_0(x, y)=1 + \delta\cos(2 \pi x) \cos(\pi y),\\
  &\mathbf{v}_0(x, y) = (1 + \delta)\left(
    (1 - \cos(2\pi x)) \sin(2\pi y), (\cos(2\pi y) - 1) \sin(2\pi x)
  \right),\\
  &c_0(x, y)= \frac{3}{4} + 0.1(1 - \delta)\cos(\pi x)\cos(\pi y),
\end{align*}
which verify \eqref{bdry_cond} and $\operatorname{div}\mathbf{v}=0$ and the density is constant in space at $\mathcal{O}(\delta)$. 

Figures \ref{fig_test_2_Cp_10_8_1}, \ref{fig_test_2_Cp_10_8_2} and \ref{fig_test_2_Cp_10_8_3} show that the order parameter $c$ evolves from outside the spinodal region toward a constant value of $3/4$.
In this state, the system follows the compressible Navier-Stokes with gravitational forces behaving in a low Mach number regime when $C_p$ becomes large. \\

\begin{figure}[h!]
    \centering
    \begin{minipage}[b]{0.45\textwidth}
        \centering
        {\small $T=0$} \\[0.3em]
        \includegraphics[width=\textwidth]{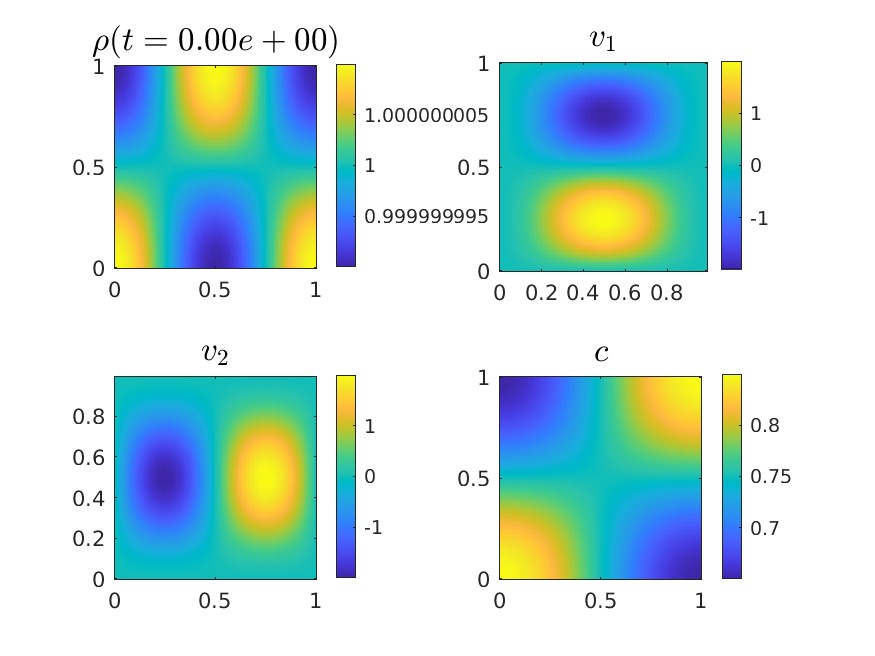}
    \end{minipage}
    \hfill
    \begin{minipage}[b]{0.45\textwidth}
        \centering
        {\small $T=0.01$} \\[0.3em]
        \includegraphics[width=\textwidth]{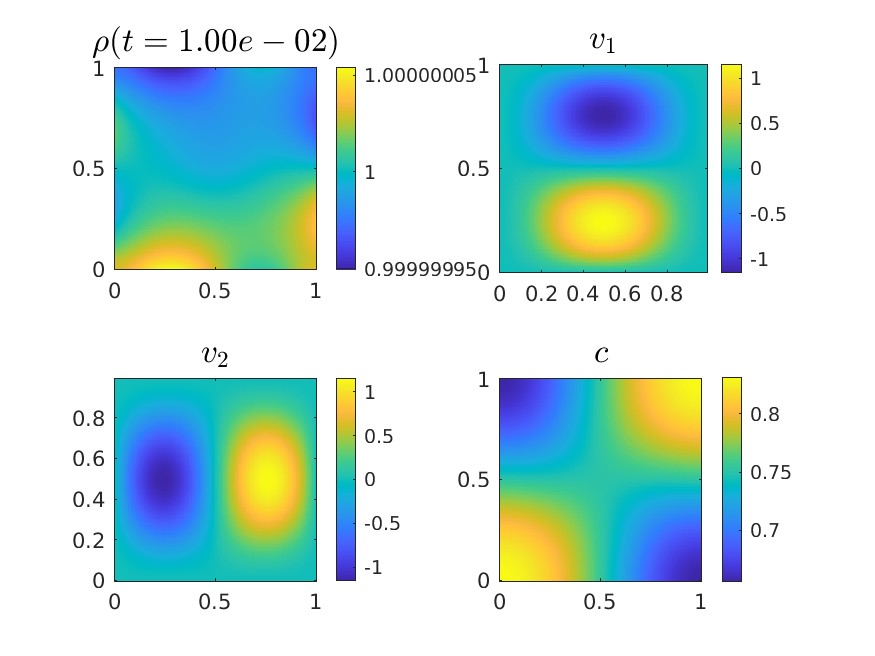}
    \end{minipage}
    \caption{Results for Test 2, $T=0, 0.01$, $M=128$ and $C_p=10^8$. 
    Initially, $c$ lies outside the unstable region and the density is dispersed. 
    At $T=0.01$, the fluid starts to have denser regions near the bottom.}
    \label{fig_test_2_Cp_10_8_1}
\end{figure}

\begin{figure}[h!]
    \centering
    \begin{minipage}[b]{0.45\textwidth}
        \centering
        {\small $T=0.03$} \\[0.3em]
        \includegraphics[width=\textwidth]{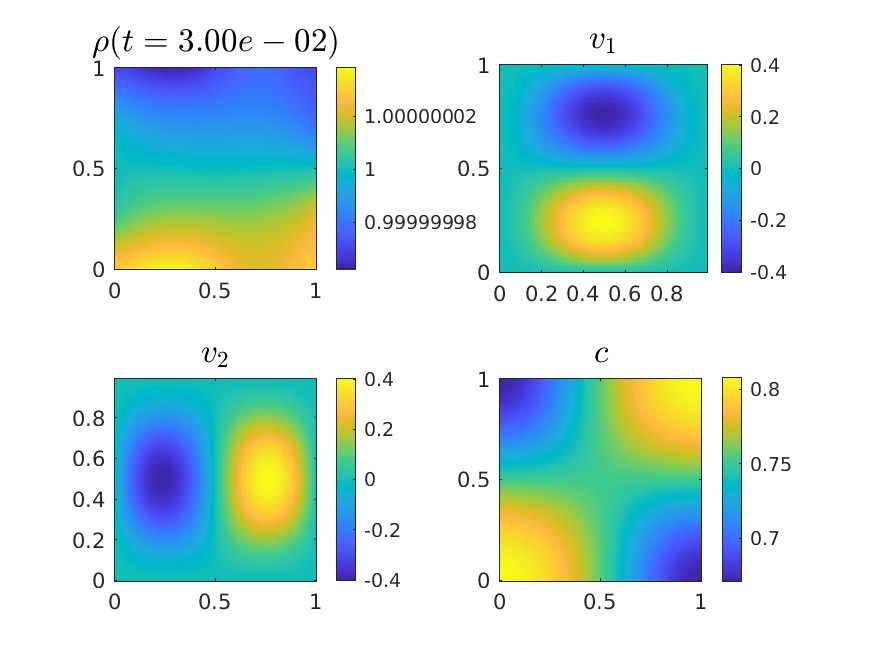}
    \end{minipage}
    \hfill
    \begin{minipage}[b]{0.45\textwidth}
        \centering
        {\small $T=0.05$} \\[0.3em]
        \includegraphics[width=\textwidth]{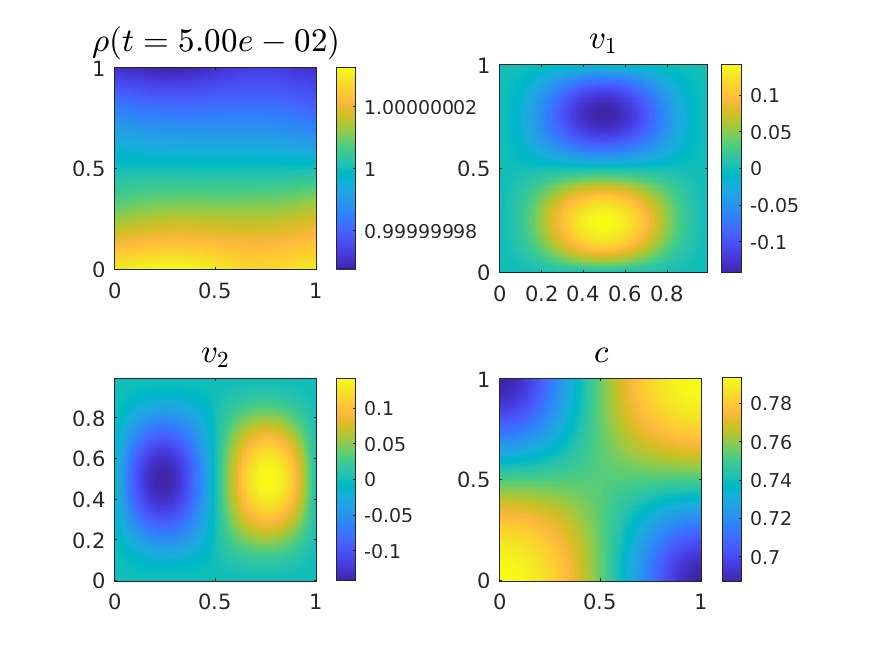}
    \end{minipage}
    \caption{Results for Test 2, $T=0.03, 0.05$, $M=128$ and $C_p=10^8$. 
    The bubbles formed order parameter start to merge around $\frac34$ also growing in size.
    The density is higher at the bottom of the domain due to gravity.}
    \label{fig_test_2_Cp_10_8_2}
\end{figure}

\begin{figure}[h!]
    \centering
    \begin{minipage}[b]{0.45\textwidth}
        \centering
        {\small $T=0.07$} \\[0.3em]
        \includegraphics[width=\textwidth]{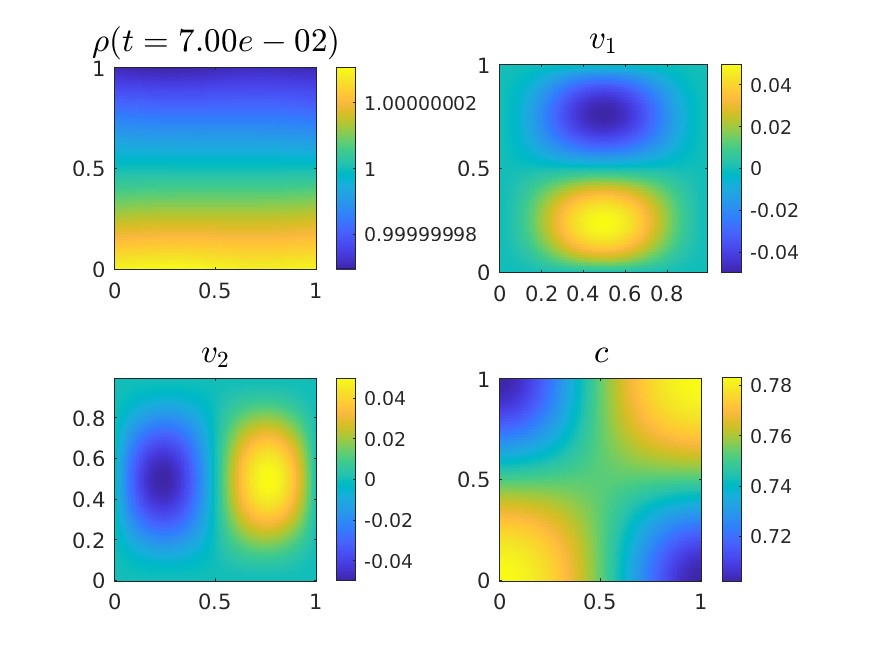}
    \end{minipage}
    \hfill
    \begin{minipage}[b]{0.45\textwidth}
        \centering
        {\small $T=0.1$} \\[0.3em]
        \includegraphics[width=\textwidth]{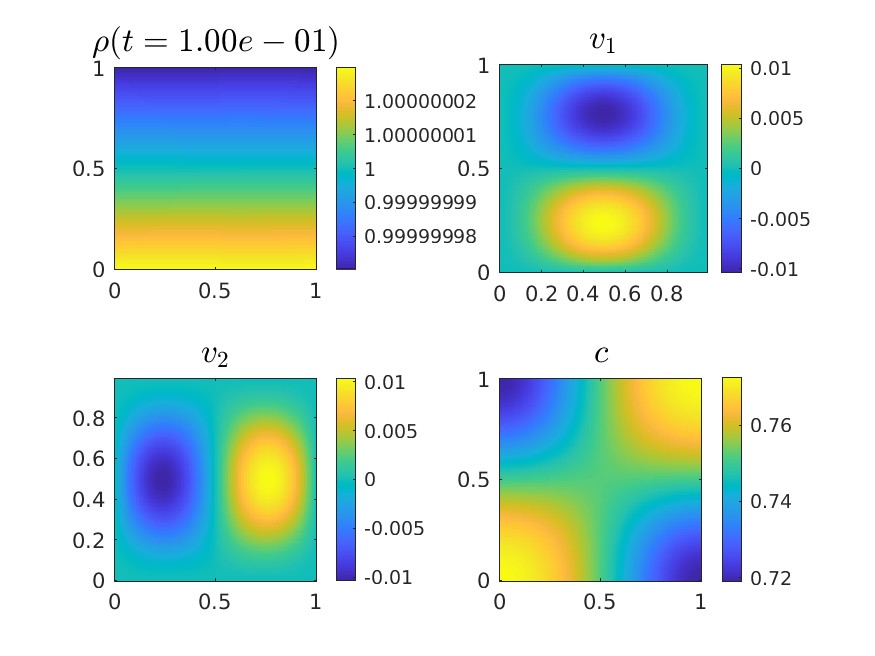}
    \end{minipage}
    \caption{Results for Test 2, $T=0.07, 0.1$, $M=128$ and $C_p=10^8$. 
    Density remains almost constant, while the order parameter starts tending to $\frac34$.}
    \label{fig_test_2_Cp_10_8_3}
\end{figure}

\textbf{Test 3} 
The aim of this test, taken from \cite{MMY25,mulet_24}, is to show the spinodal decomposition.
To this end, the initial conditions are set as $\rho_0=1$, $\mathbf{v}_0=0$, and $c_0$ is initialized as a uniform random sample of zero mean
and $10^{-10}$ standard deviation.

Figures \ref{fig_test_3_Cp_10_8_1} and \ref{fig_test_3_Cp_10_8_2} show the results up to $T=0.1$, where the spinodal decomposition occurs at the beginning of the simulations. 
In addition, density becomes higher near the bottom boundary $y=0$ due to the gravitational effects, but it seems to be stabilized as time evolves, see Figure \ref{fig_test_3_Cp_10_8_2}.

\begin{figure}[h!]
    \centering
    \begin{minipage}[b]{0.45\textwidth}
        \centering
        {\small $T=0$} \\[0.3em]
        \includegraphics[width=\textwidth]{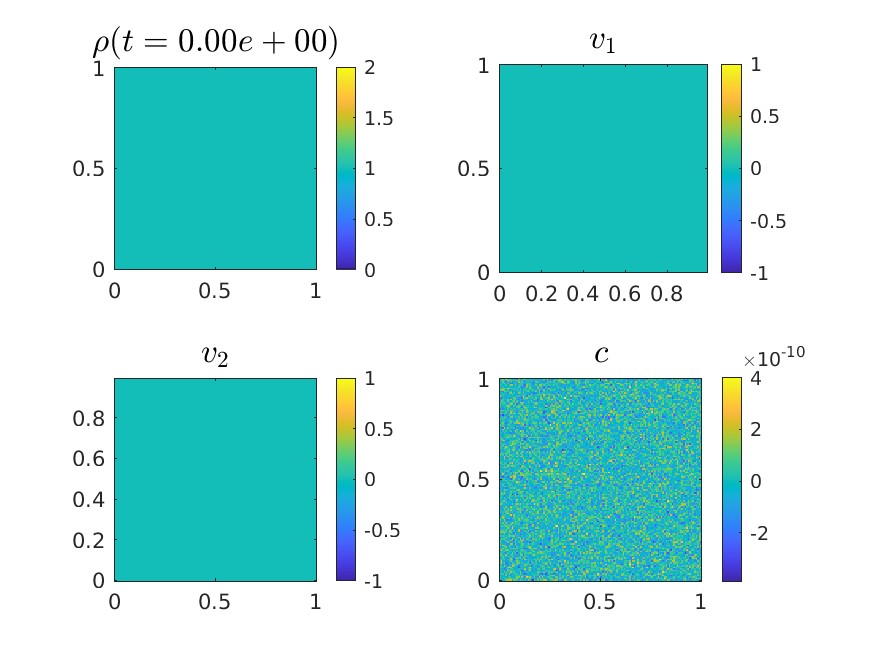}
    \end{minipage}
    \hfill
    \begin{minipage}[b]{0.45\textwidth}
        \centering
        {\small $T=0.01$} \\[0.3em]
        \includegraphics[width=\textwidth]{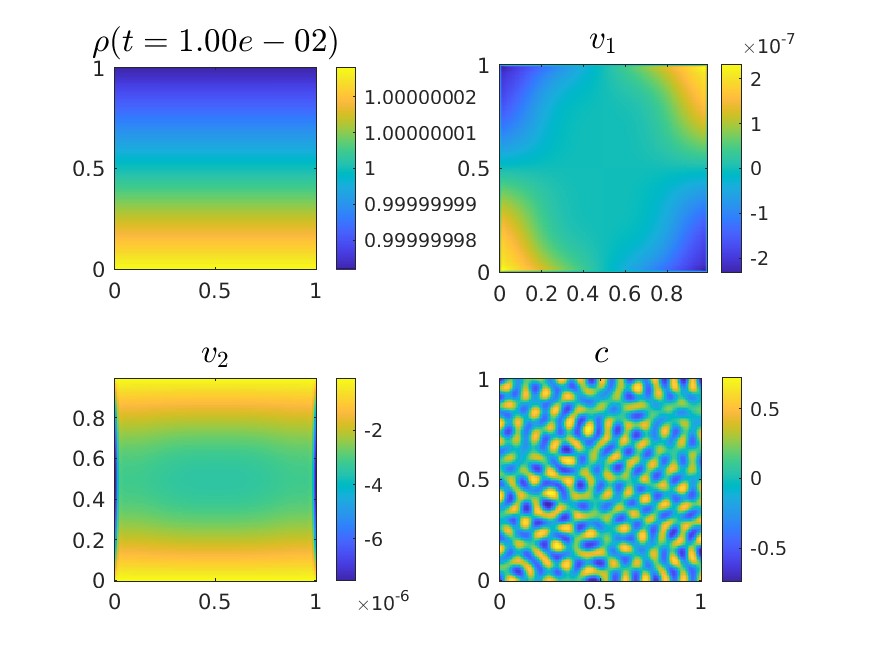}
    \end{minipage}
    \caption{Results for Test 3, $T=0, 0.01$, $M=128$ and $C_p=10^8$. 
    Initially, density is constant, velocity is zero everywhere, and the order parameter is a random perturbation around $c=0$.
    It can be seen at $T=0.01$ that phase separation has started to occur, while density is higher at the bottom.}
    \label{fig_test_3_Cp_10_8_1}
\end{figure}

\begin{figure}[h!]
    \centering
    \begin{minipage}[b]{0.45\textwidth}
        \centering
        {\small $T=0.03$} \\[0.3em]
        \includegraphics[width=\textwidth]{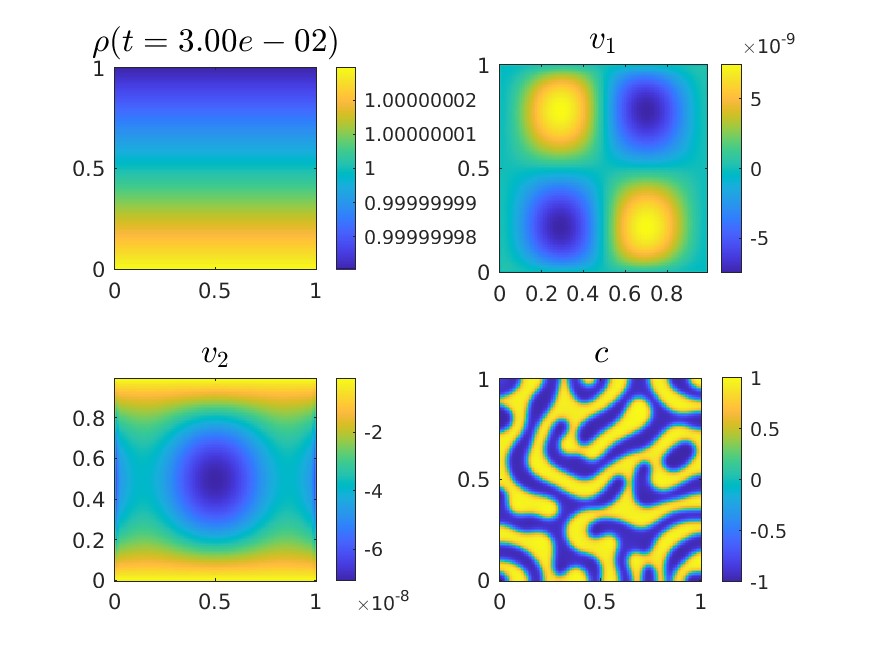}
    \end{minipage}
    \hfill
    \begin{minipage}[b]{0.45\textwidth}
        \centering
        {\small $T=0.05$} \\[0.3em]
        \includegraphics[width=\textwidth]{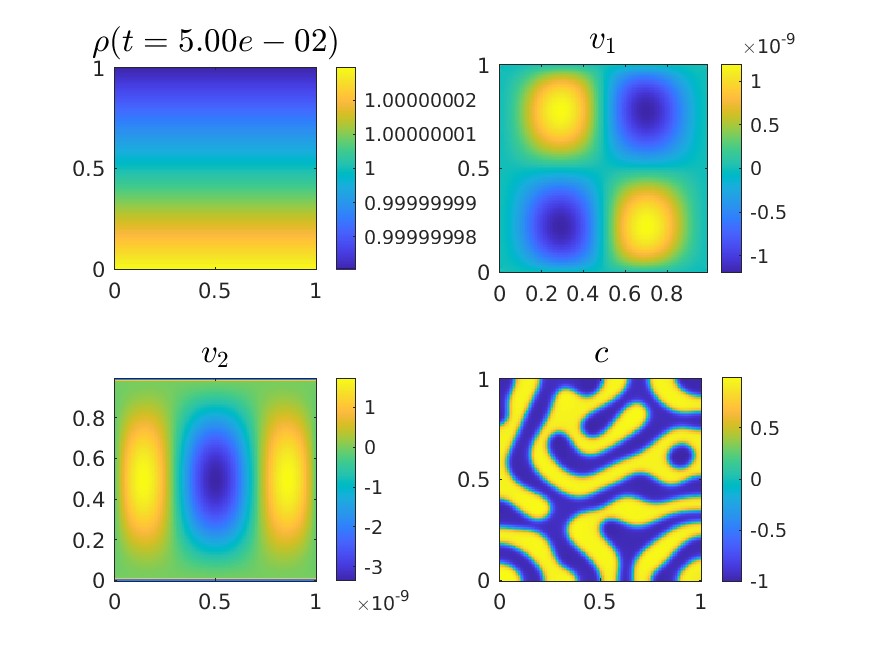}
    \end{minipage}
    \\[0.5em]
    \begin{minipage}[b]{0.45\textwidth}
        \centering
        {\small $T=0.07$} \\[0.3em]
        \includegraphics[width=\textwidth]{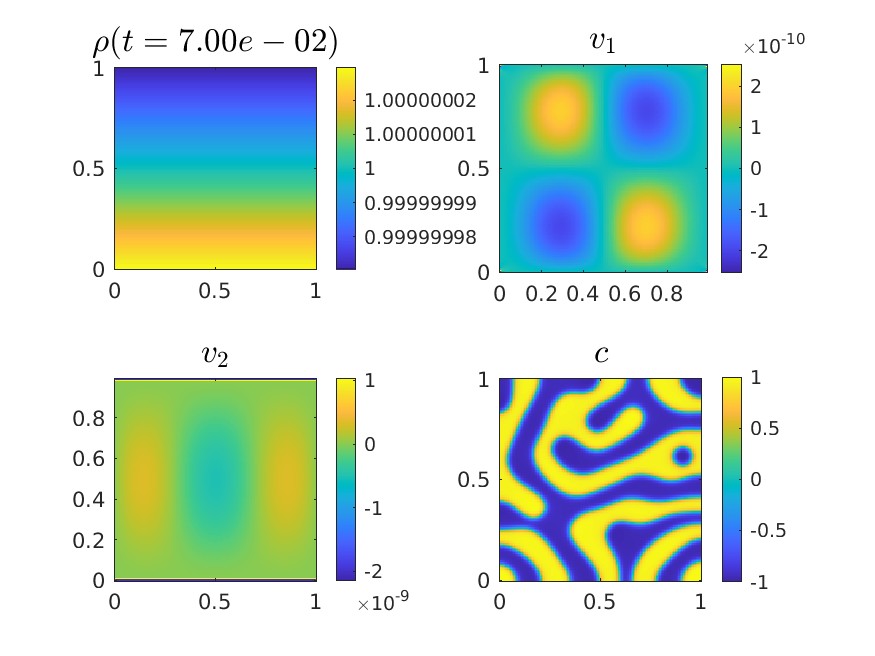}
    \end{minipage}
    \hfill
    \begin{minipage}[b]{0.45\textwidth}
        \centering
        {\small $T=0.1$} \\[0.3em]
        \includegraphics[width=\textwidth]{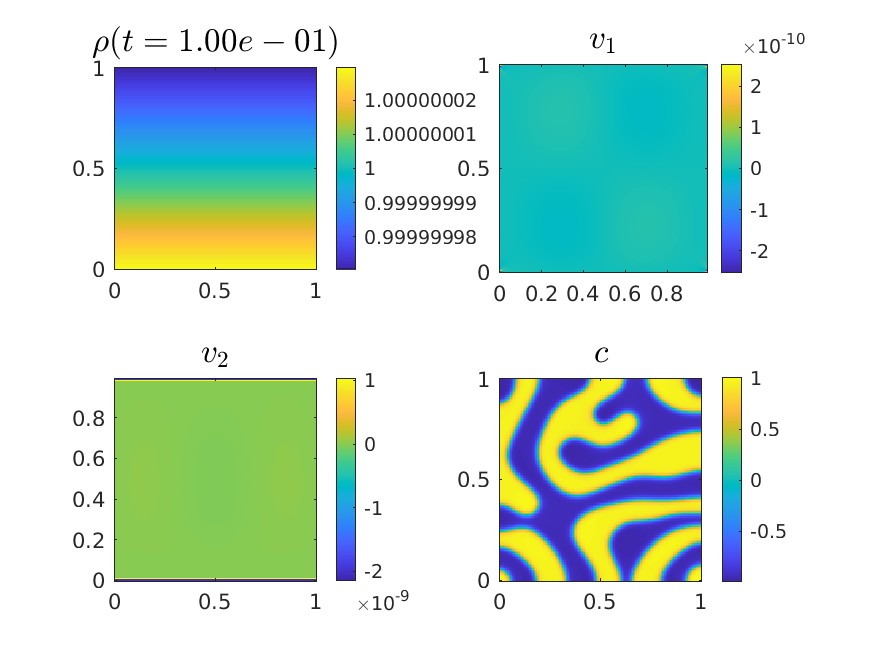}
    \end{minipage}
    \caption{Results for Test 3, $T=0.03, 0.05, 0.07, 0.1$, $M=128$ and $C_p=10^8$. 
    The process of phase separation continues, and the density seems to stabilize with higher values at the bottom of the domain.}
    \label{fig_test_3_Cp_10_8_2}
\end{figure}

\subsubsection{Conservation of mass and bound-preserving properties}
Figure \ref{fig_err_mass} illustrates the conservation of mass for the three tests.
Specifically, the mass conservation errors for $\rho$ and $q=\rho c$ are computed using
\begin{equation*}
  \begin{split}
    \int_{\Omega}\!\rho(\mathbf{x},t_n)\ d\mathbf{x} - \int_\Omega\!\rho(\mathbf{x},0)\ d\mathbf{x}
    &= \sum_{k=1}^{M^2}\!\varrho_k^n - \sum_{k=1}^{M^2}\!\varrho^0_k,\\[5pt]
    \int_{\Omega}\!q(\mathbf{x},t_n)\ d\mathbf{x} - \int_\Omega\!q(\mathbf{x},0)\ d\mathbf{x}
    &= \sum_{k=1}^{M^2}\!(\varrho\ast C)_k^n - \sum_{k=1}^{M^2}\!(\varrho\ast C)^0_k.
  \end{split}
\end{equation*}
On the other hand, the order parameter $c$ has not exceeded $[-1,1]$ considerably, and it has been kept below its bounds throughout all experiments. 
Table \ref{table_c_dispersion} shows the maxim and minimum value of $c$ during the performance.
Figure \ref{fig_c_evol} shows the time evolution of the maximum and minimum values of the $c$-component, which rarely exceed the interval $[-1,1]$.
Therefore, the chosen CFL number of $0.4$ can be considered safe for our simulations.

\begin{figure}[h!]
    \centering
    \begin{minipage}[b]{0.30\textwidth}
        \centering
        \includegraphics[width=\textwidth]{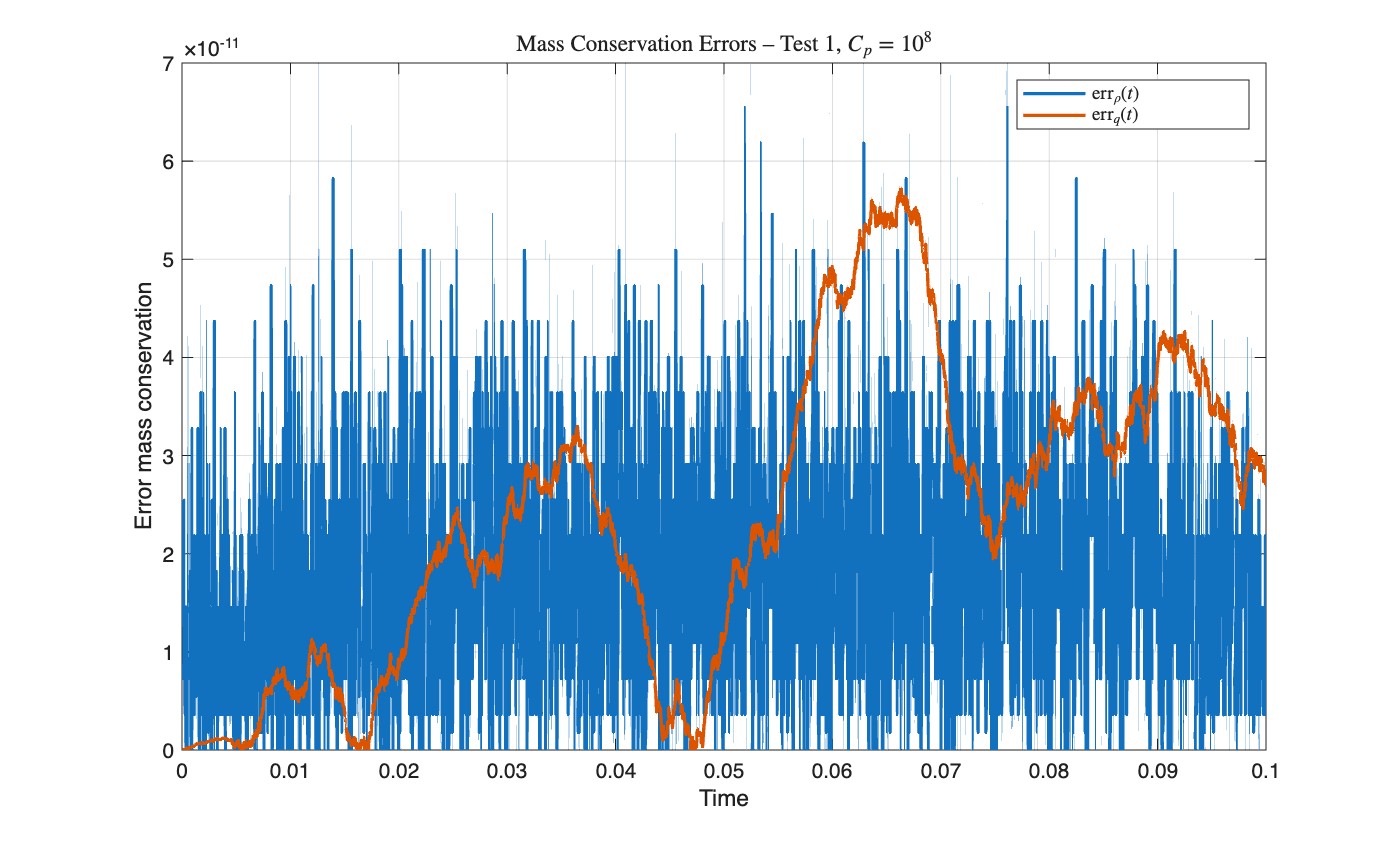}
    \end{minipage}
    \hfill
    \begin{minipage}[b]{0.30\textwidth}
        \centering
        \includegraphics[width=\textwidth]{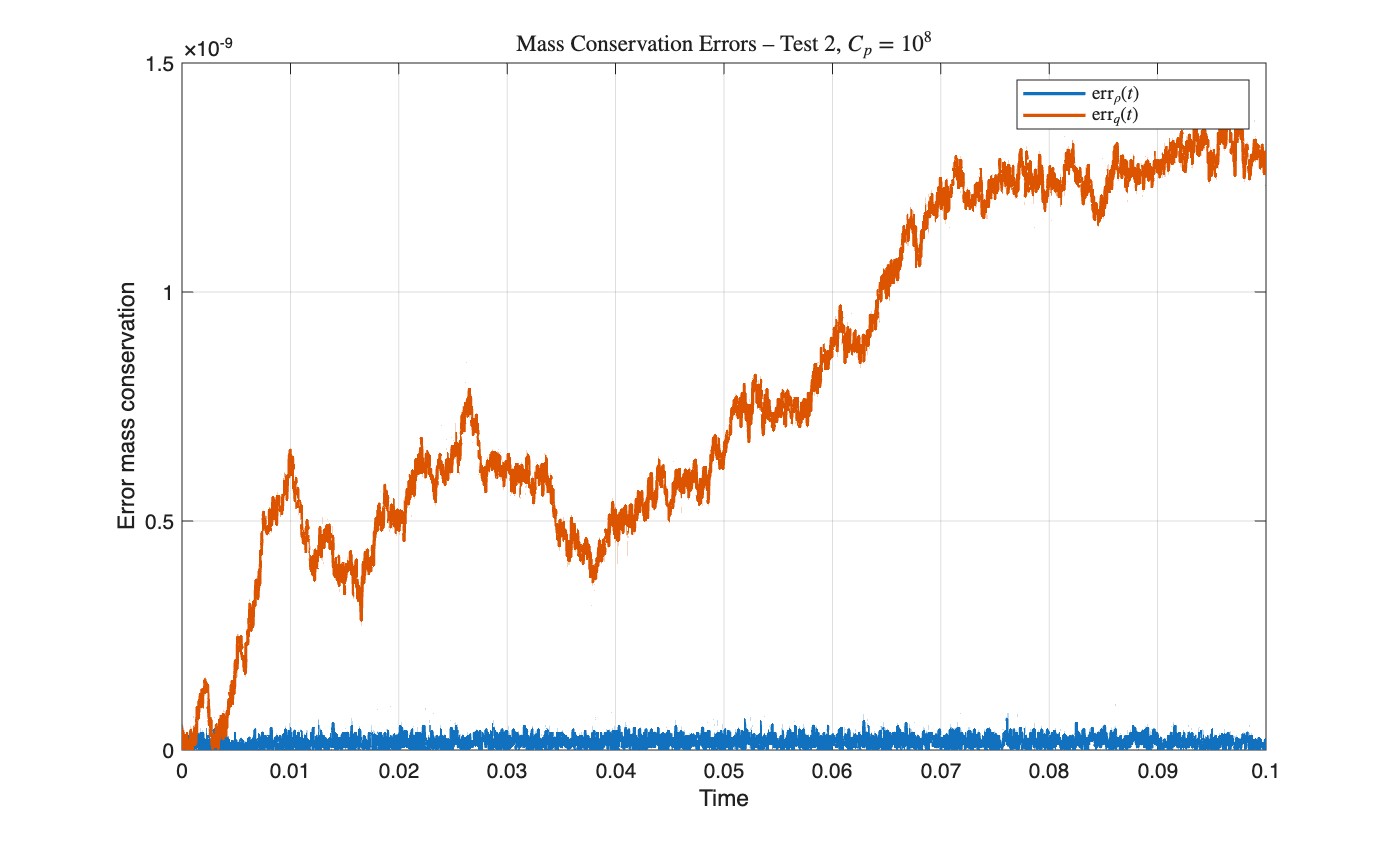}
    \end{minipage}
    \hfill
    \begin{minipage}[b]{0.30\textwidth}
        \centering
        \includegraphics[width=\textwidth]{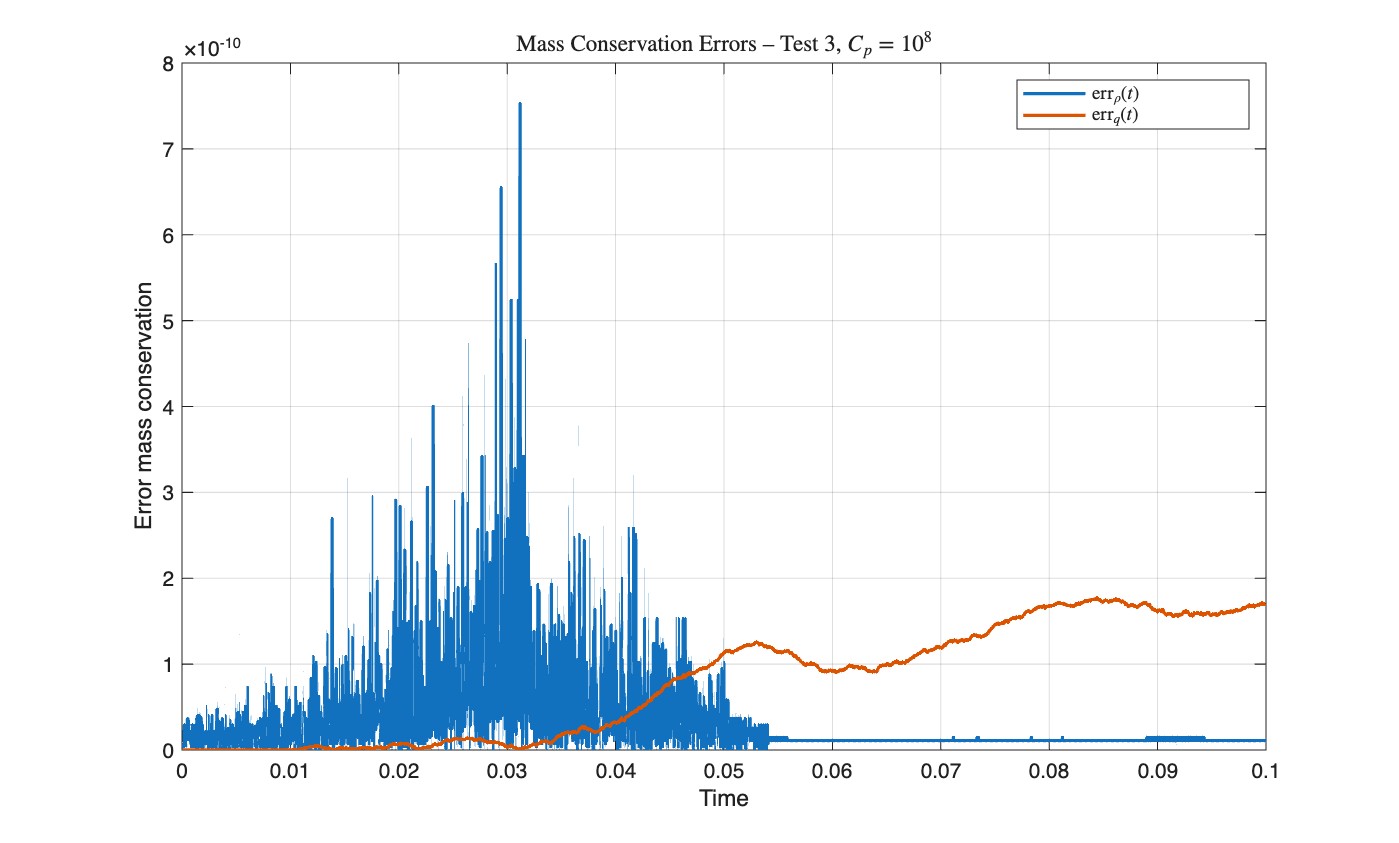}
    \end{minipage}
    \caption{Time evolution of the mass conservation errors for both $\rho$ and $q$ with $M=128$ for Test 1, 2, and 3 for $C_p=10^8$.}
    \label{fig_err_mass}
\end{figure}

\begin{figure}[h!]
    \centering
    \begin{minipage}[b]{0.30\textwidth}
        \centering
        \includegraphics[width=\textwidth]{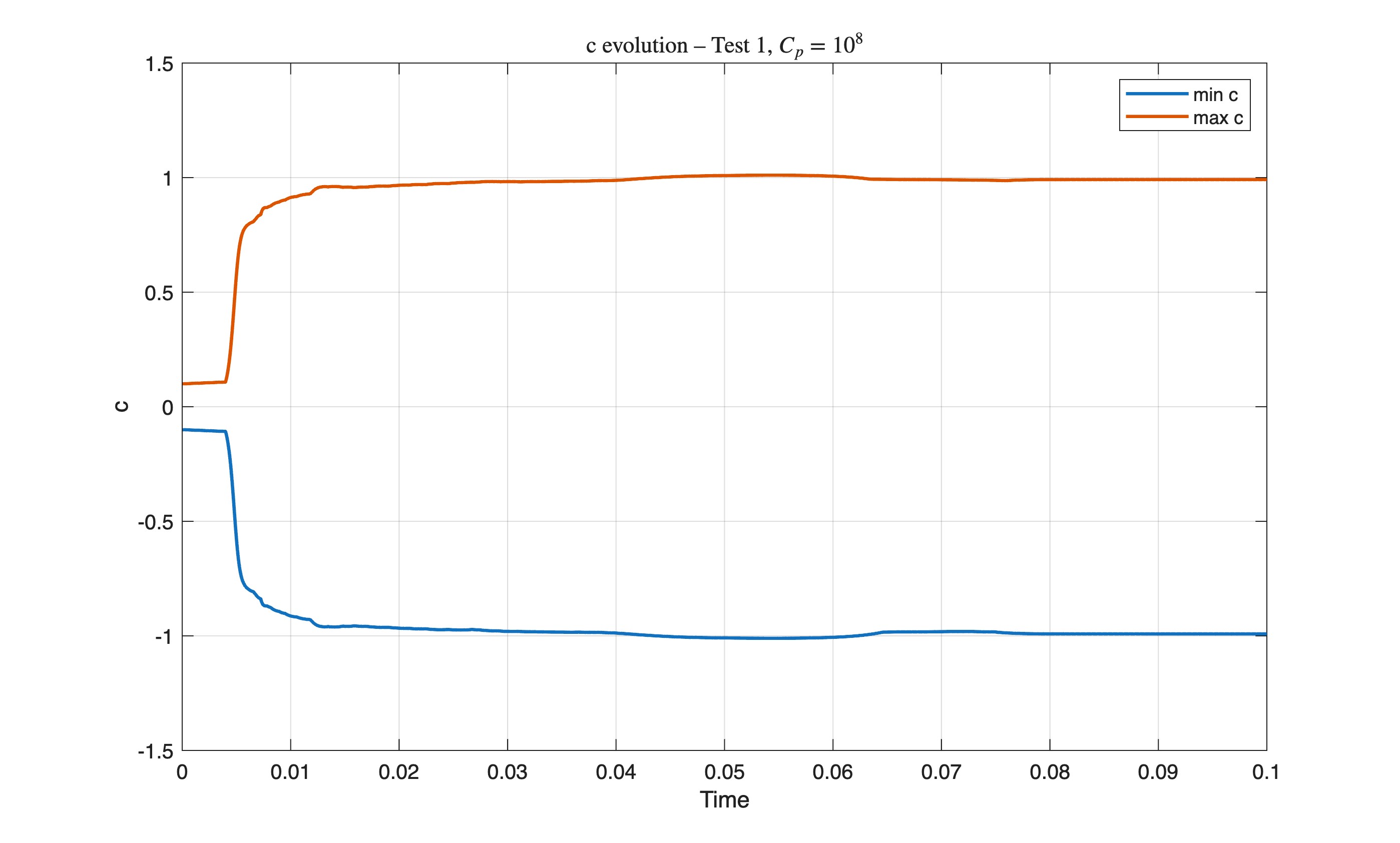}
    \end{minipage}
    \hfill
    \begin{minipage}[b]{0.30\textwidth}
        \centering
        \includegraphics[width=\textwidth]{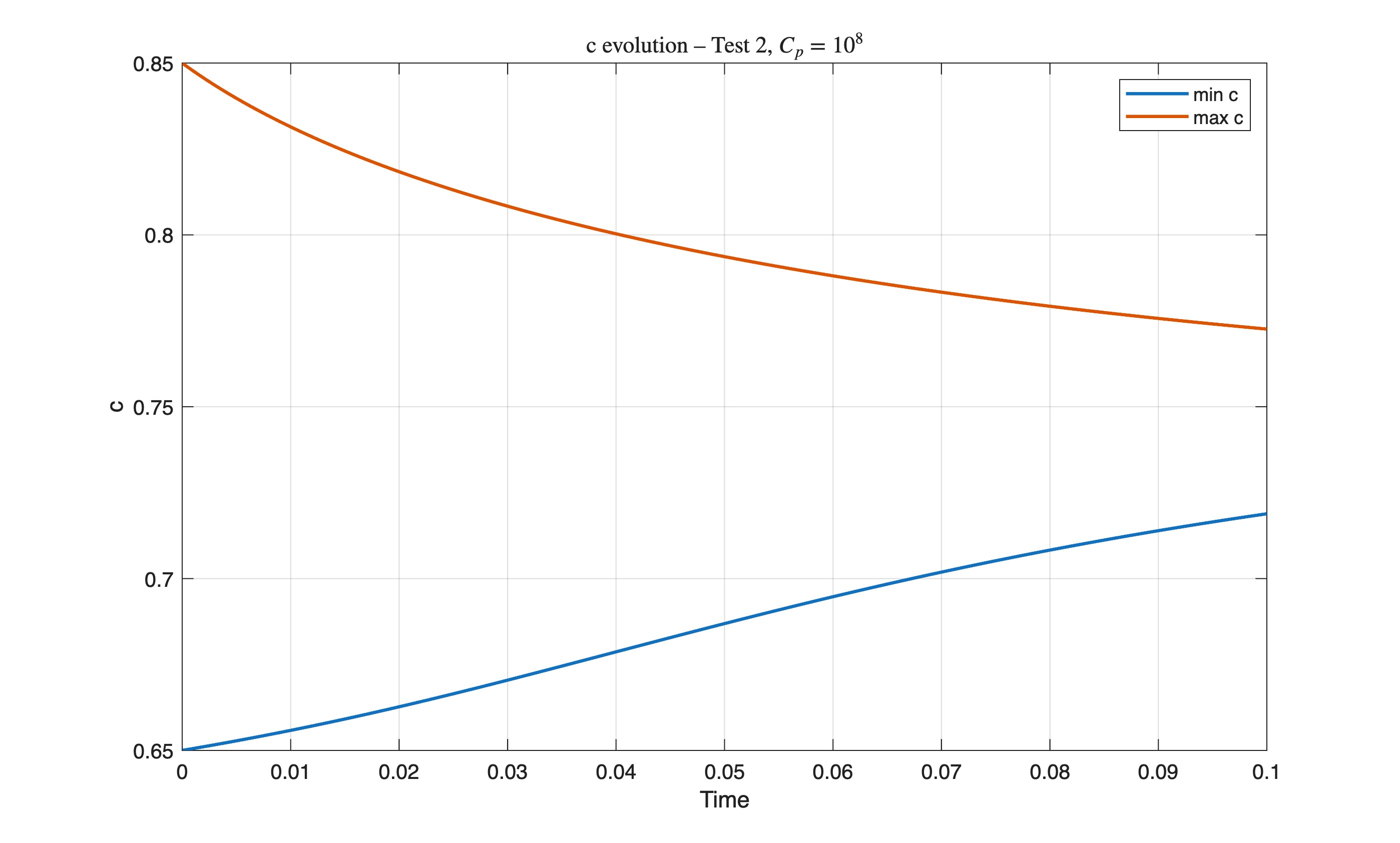}
    \end{minipage}
    \hfill
    \begin{minipage}[b]{0.30\textwidth}
        \centering
        \includegraphics[width=\textwidth]{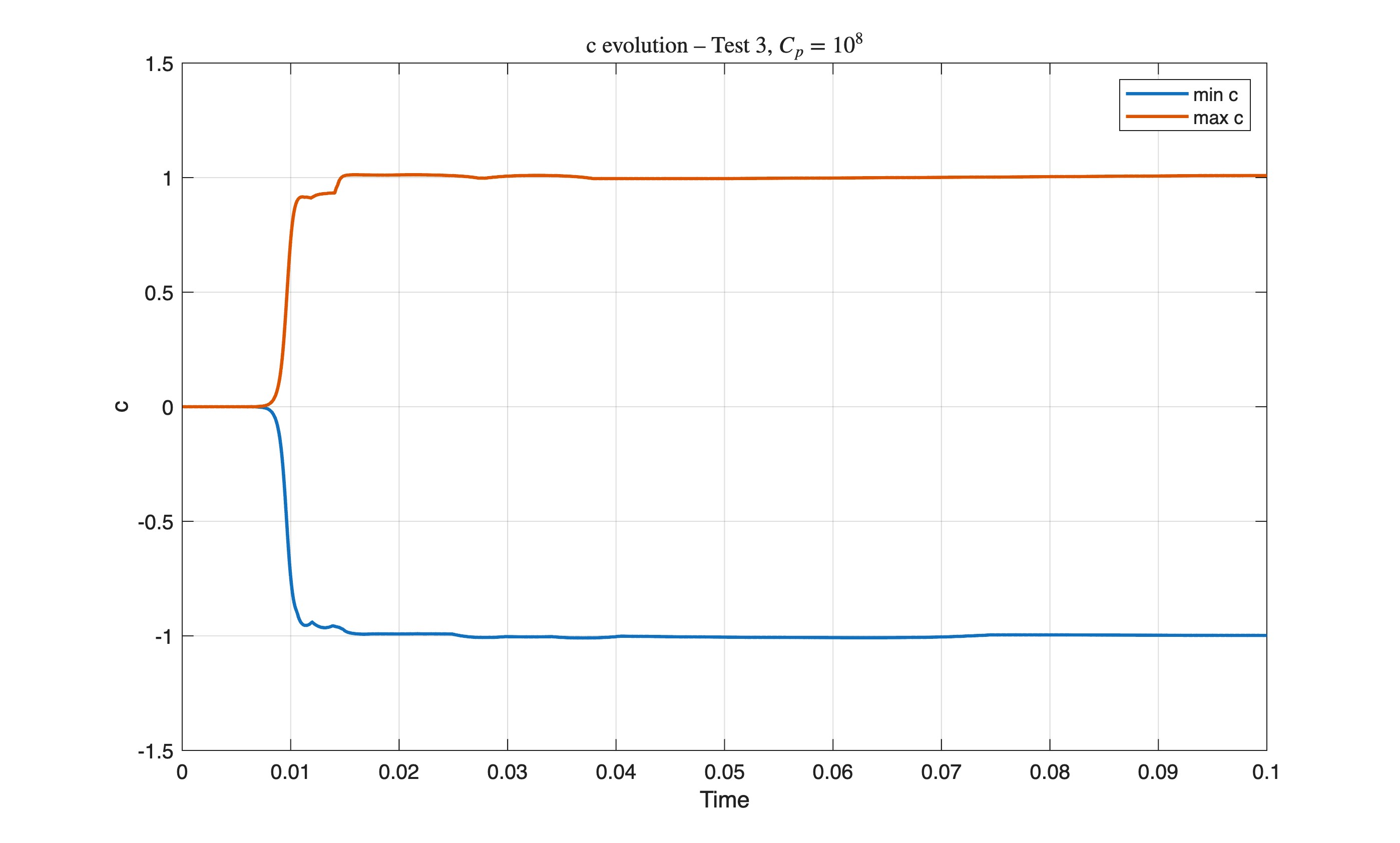}
    \end{minipage}

    \caption{Time evolution of the maximum and minimum values of the order parameter $c$ with $M=128$ for Test~1, 2, and 3 for $C_p=10^8$.}
    \label{fig_c_evol}
\end{figure}

\begin{table}
\centering
\begin{tabular}{c|cccc}
\hline
 & $C_p = 10^2$ & $C_p = 10^4$ & $C_p = 10^6$ & $C_p = 10^8$ \\
\hline
Test 1 & 
$\begin{matrix} 0.9976 \\ -0.9924 \end{matrix}$ &
$\begin{matrix} 1.0130 \\ -1.0128 \end{matrix}$ &
$\begin{matrix} 1.0113 \\ -1.0114 \end{matrix}$ &
$\begin{matrix} 1.0103 \\ -1.0103 \end{matrix}$ \\[6pt]
\hline
Test 3 &
$\begin{matrix} 1.0097 \\ -1.0100 \end{matrix}$ &
$\begin{matrix} 1.0121 \\ -1.0100 \end{matrix}$ &
$\begin{matrix} 1.0112 \\ -1.0128 \end{matrix}$ &
$\begin{matrix} 1.0122 \\ -1.0092 \end{matrix}$ \\
\hline
\end{tabular}
\caption{Maximum and minimum values for the $c$ evolution in Test 1 and 3.}
\label{table_c_dispersion}
\end{table}

\subsubsection{Low Mach number regime}
Here, we analyze the limit of the compressible scheme $\mathcal{M}^\delta_\Delta$ toward the incompressible scheme $\mathcal{M}^0_\Delta$.
To this end, we focus on the three previous tests specially when the squared Mach numbers are $\delta=10^{-2k}$ for $k=2,3,4$.
Note that in these tests, the initial conditions are well-prepared in the sense of \eqref{eq_well_prepared_data}.
In addition, according to Theorem \ref{teo_AP}, each time-step must be well-prepared, ensuring that the scheme is AP.
Figures \ref{fig_div_free_test_1_2_3} and \ref{fig_incompressibility_prop_test} illustrate this behavior: the density approaches to $1$, and the divergence free condition is satisfied.

\begin{figure}[h!]
  \centering
  \begin{minipage}[b]{0.32\textwidth}
        \centering
        \includegraphics[width=\textwidth]{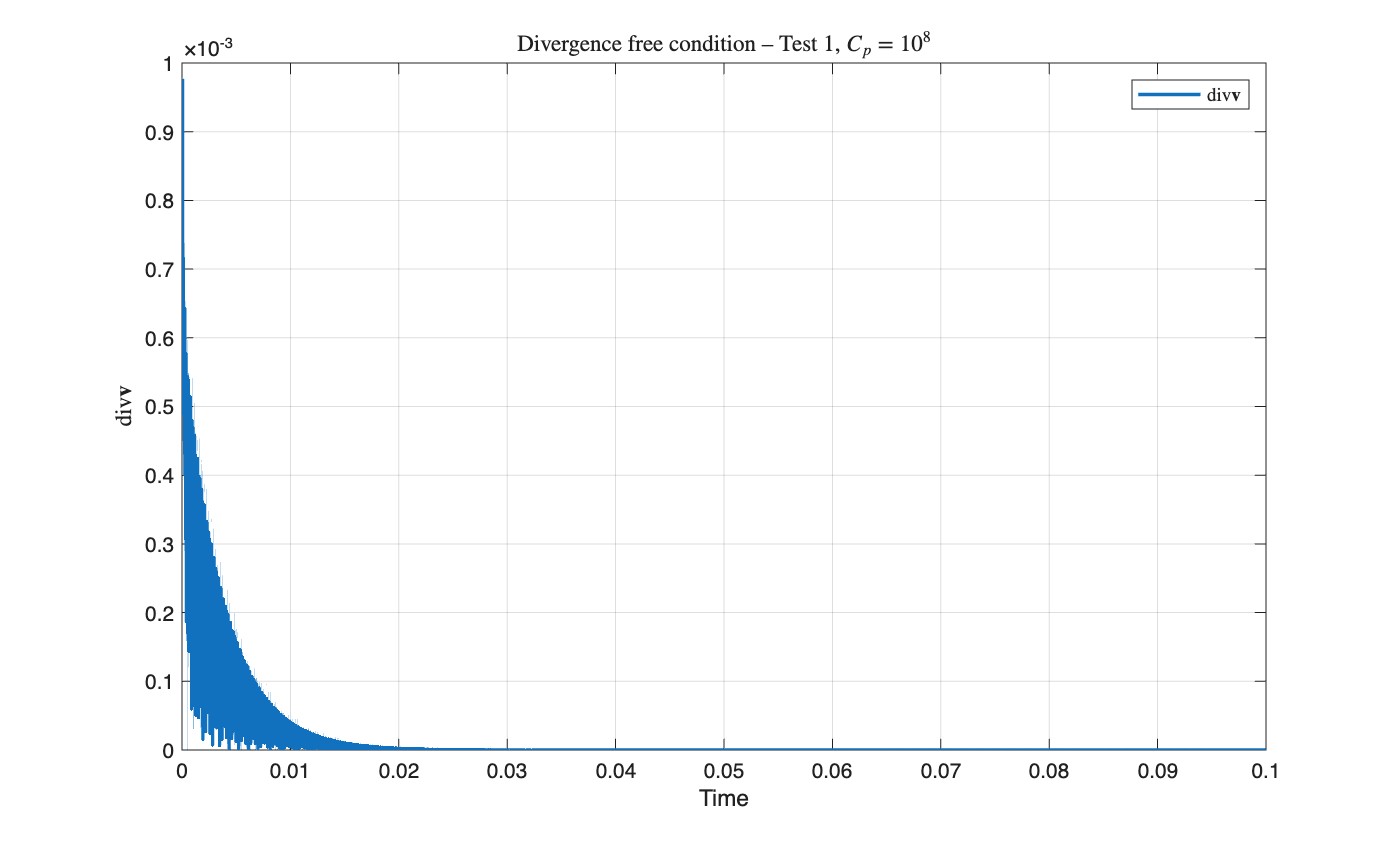}
    \end{minipage}
    \hfill
    \begin{minipage}[b]{0.32\textwidth}
        \centering
        \includegraphics[width=\textwidth]{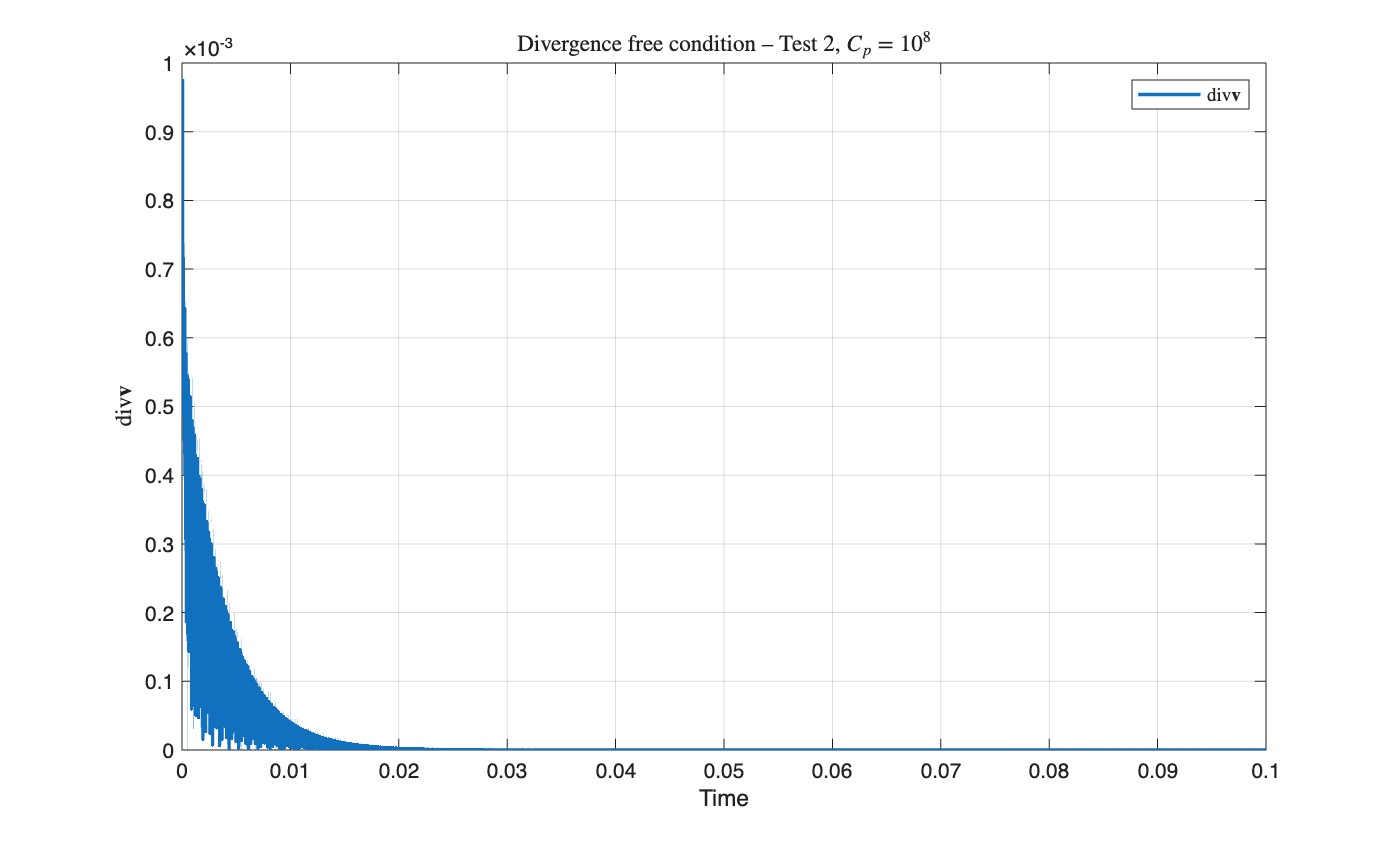}
    \end{minipage}
    \hfill
    \begin{minipage}[b]{0.32\textwidth}
        \centering
        \includegraphics[width=\textwidth]{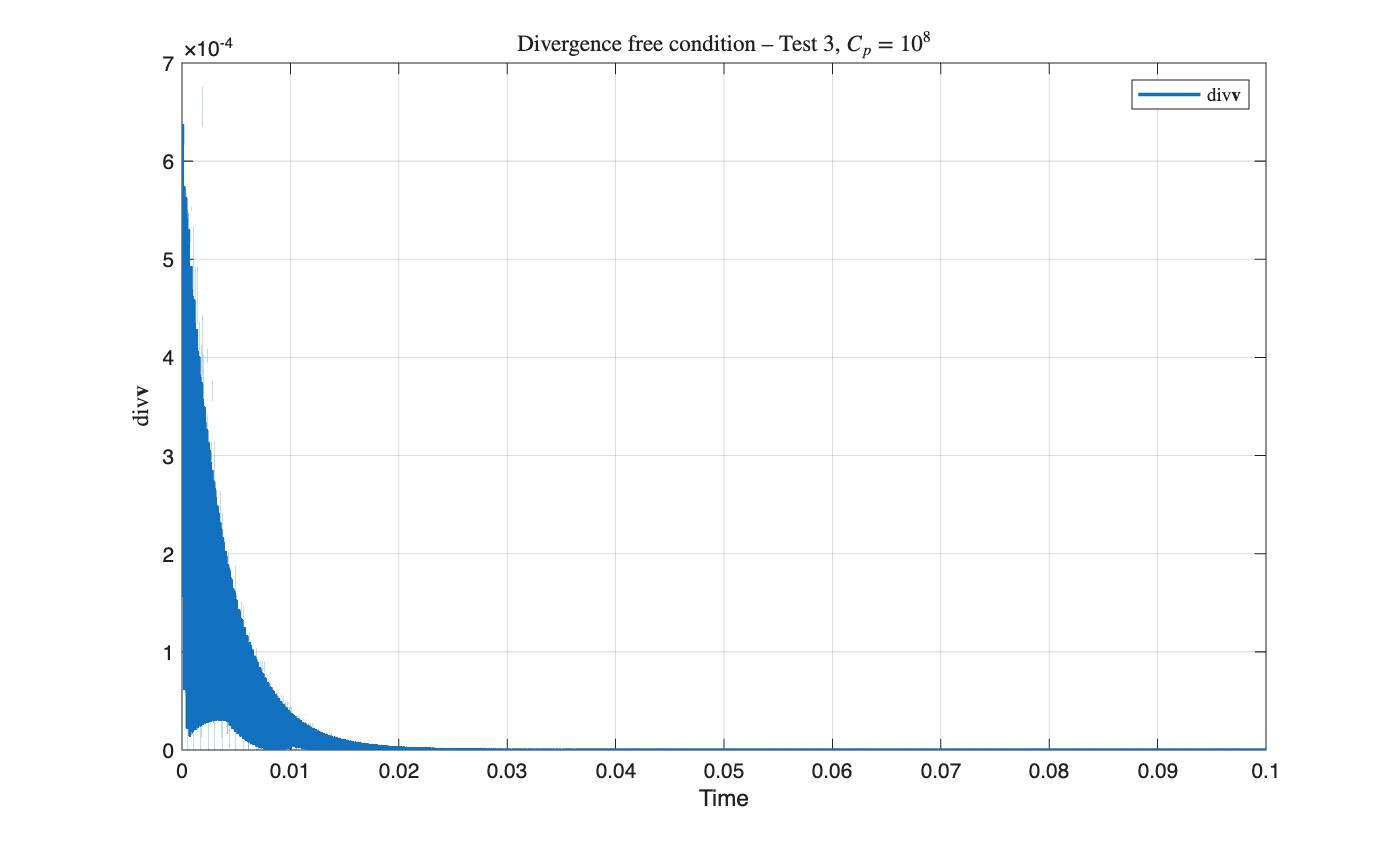}
    \end{minipage}
    \caption{Divergence free condition for Test 1, Test 2 and Test 3 with $M=128$ and $C_p=10^{8}$.}
    \label{fig_div_free_test_1_2_3}
\end{figure}

\begin{figure}[h!]
  \centering
    \begin{minipage}[b]{0.32\textwidth}
        \centering
        \includegraphics[width=\textwidth]{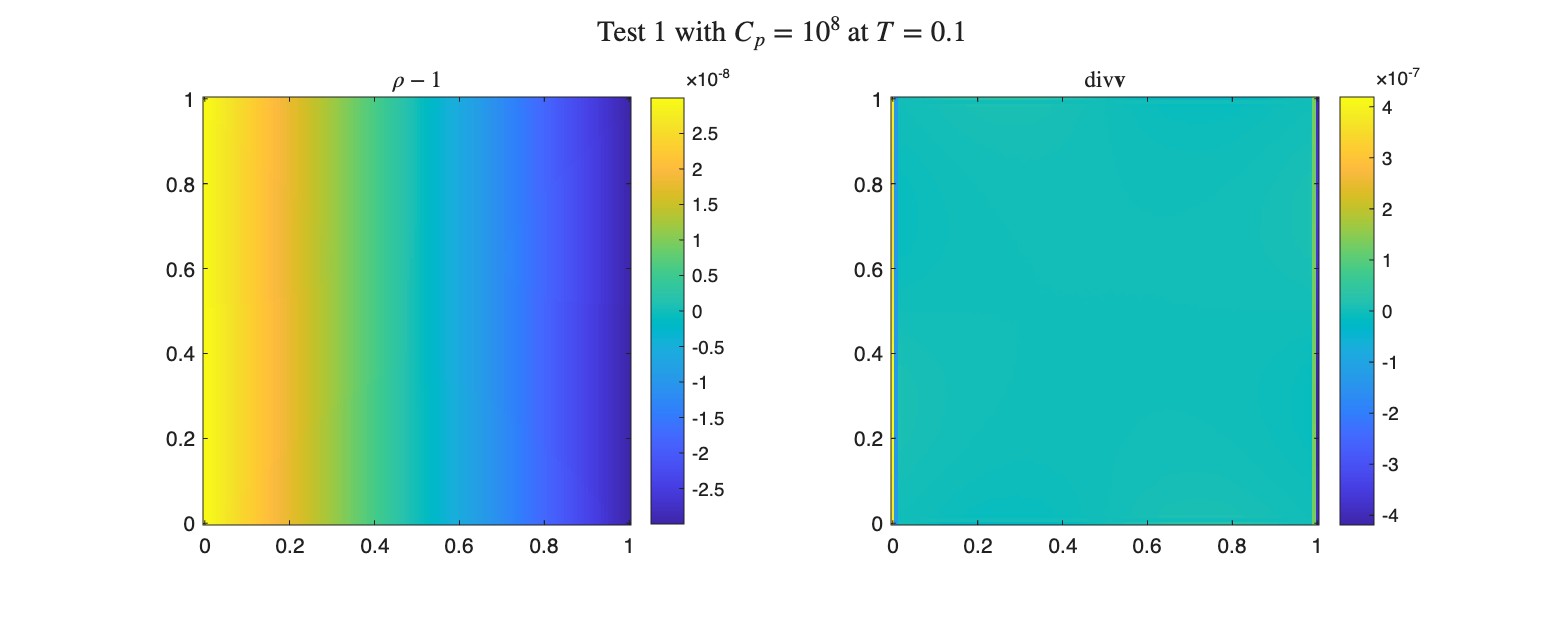}
    \end{minipage}
    \hfill
    \begin{minipage}[b]{0.32\textwidth}
        \centering
        \includegraphics[width=\textwidth]{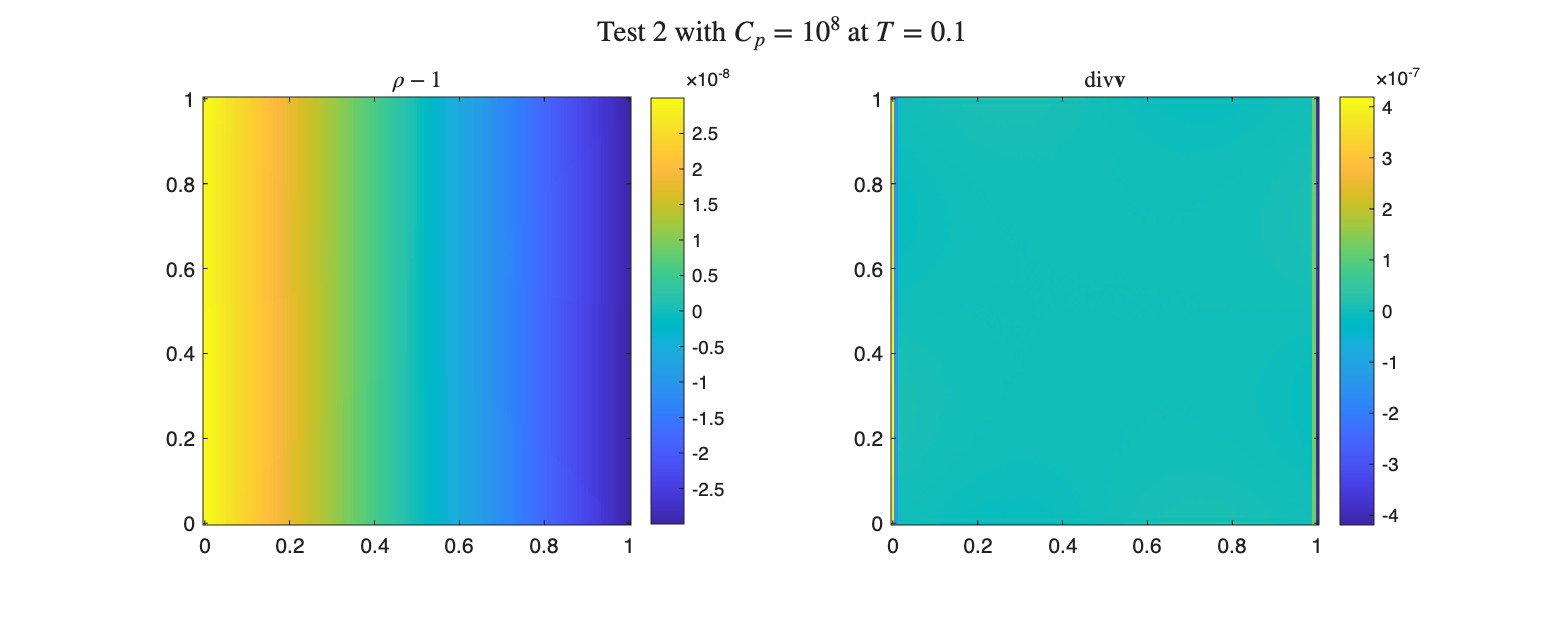}
    \end{minipage}
    \hfill
    \begin{minipage}[b]{0.32\textwidth}
        \centering
        \includegraphics[width=\textwidth]{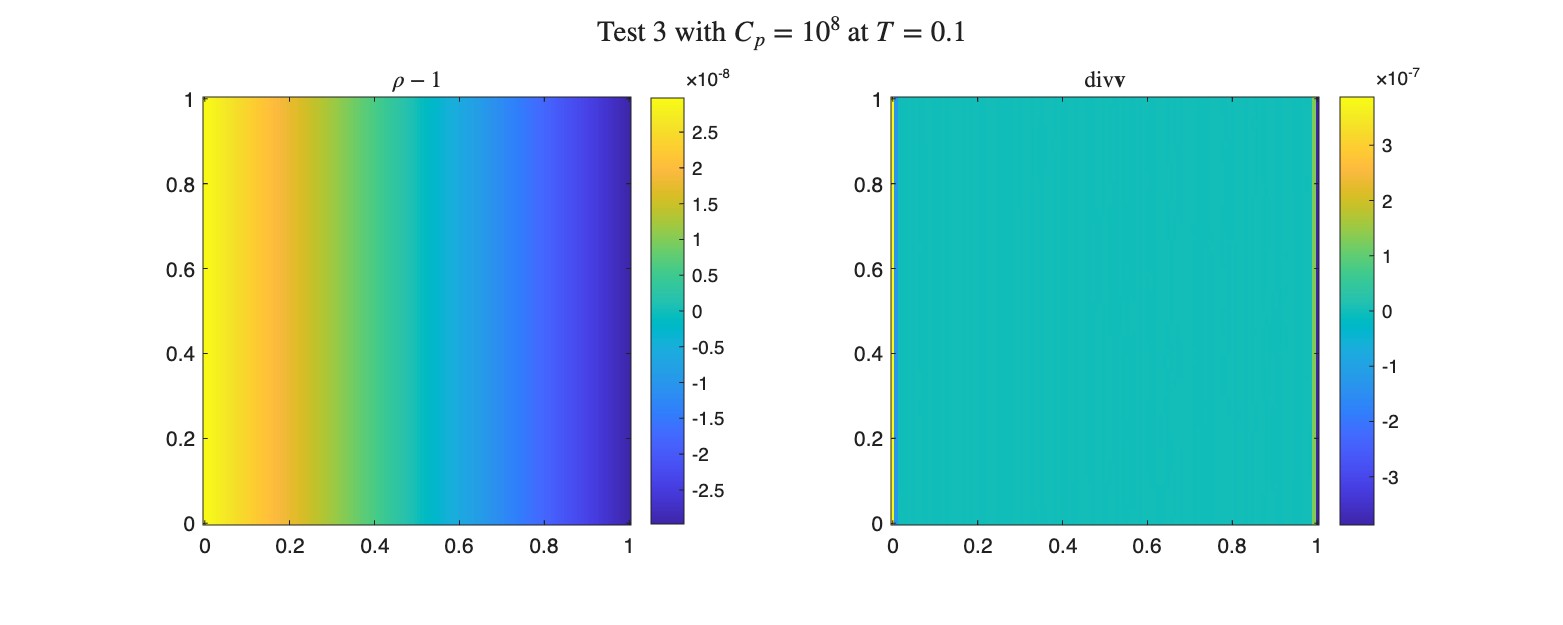}
    \end{minipage}
    \caption{Well-preparedness of the solution for Test 1, 2, and 3 at $T=0.1$ with $M=128$ and $C_p=10^{8}$.}
    \label{fig_incompressibility_prop_test}
\end{figure}

\section{Conclusions and future work}\label{section_conlusion}
In this work, we present an efficient second-order asymptotic-preserving IMEX schemes on staggered grids for the two-dimensional compressible isentropic Cahn-Hilliard-Navier-Stokes equations for any Mach number regime.
The proposed method avoids the severe restriction imposed by the high-order and stiff pressure terms.
To validate the method, several numerical test have been performed, showing that second-order accuracy is achieved with the time-step constrained only by the convective subsystem of the equations.

For future work, we aim to extend the present framework to the non-barotropic compressible Cahn-Hilliard-Navier-Stokes in a low Mach number regime, as well as to the three-dimensional case using Galerkin techniques.

Regarding the pressure splitting defined in \eqref{pressure_splitting}, we plan to further research on the possible range of $C_{p,i}$ values. 
Our current strategy of setting $C_{p,1}=\sqrt{C_p}$ has proven to be successful in our experiments, although no formal proof is given.

When solving systems \eqref{nonlinear_subsystem} and \eqref{system_c}, neither the positivity of the density $\rho$ nor the boundedness of the order parameter $c\in[-1,1]$ can be guaranteed.
To address this issue, we plan to employ bound-preserving high-order reconstructions schemes which can effectively circumvent these physical constraints.

\subsection*{Conflict of interest} The authors declare that they have no conflict of interest.

\subsection*{Data Availability Statements} Data sharing is not applicable to this article as no datasets were generated or analyzed during the current study.

\section*{Acknowledgments}
This paper has received financial support from the research projects
PID2023-146836NB-I00, granted by MCIN/ AEI /10.13039/ 501100011033, and 
CIAICO/2024/089, granted by GVA.

\end{document}